\documentclass[12pt, leqno, british, final]{amsart}
\usepackage[
	style=alphabetic,
	citestyle=alphabetic,
	backend=biber
]{biblatex}
\usepackage{a4, amsmath}
\usepackage{mathtools}
\usepackage{amssymb}
\usepackage{amsthm, amscd, mathdots}
\swapnumbers
\usepackage{enumerate}
\usepackage{url}
\usepackage{hyperref}
\usepackage{cleveref}
\usepackage{csquotes}
\usepackage{color}
\usepackage{datetime}
\usepackage{todonotes}
\usepackage{tikz-cd}

\usepackage{microtype}
\usepackage[inline]{showlabels}

\usepackage[pass]{geometry}

\usepackage[all]{xy}

\theoremstyle{definition}
\newtheorem{defi}{Definition}[section]

\theoremstyle{plain}
\newtheorem{prop}[defi]{Proposition}
\newtheorem{lem}[defi]{Lemma}
\newtheorem{stel}[defi]{Theorem}
\newtheorem{gev}[defi]{Corollary}

\newtheorem*{stel*}{Theorem}
\theoremstyle{remark}
\newtheorem{opm}[defi]{Remark}

\newtheorem{vbn}[defi]{Examples}

\newcommand{\nat}{\mathbb{N}}
\newcommand{\zz}{\mathbb{Z}}
\newcommand{\rr}{\mathbb{R}}

\newcommand{\cc}{\mathbb{C}}
\newcommand{\qq}{\mathbb{Q}}
\newcommand{\Z}{\mathbb{Z}}
\newcommand{\mc}{\mathcal}
\newcommand{\mf}{\mathfrak}

\newcommand{\pow}[1]{^{(#1)}}

\newcommand{\Val}{\mathcal{V}}

\newcommand{\Lar}{{\mathcal{L}_{\rm ring}}}

\newcommand{\llangle}{\langle\!\langle}
\newcommand{\rrangle}{\rangle\!\rangle}

\DeclareMathOperator{\cd}{cd}

\DeclareMathOperator{\charac}{char}

\DeclareMathOperator{\Odd}{\mathsf{Odd}}
\DeclareMathOperator{\Neg}{\mathsf{Neg}}
\DeclareMathOperator{\Pic}{Pic}
\DeclareMathOperator{\Hom}{Hom}
\newcommand{\RF}[2]{\mathrm{r}_{#1}(#2)}

\newcommand{\ovl}{\overline}

\DeclareMathOperator{\Coker}{Coker}

\title{Universally defining subrings in function fields}
\author{Nicolas Daans and Philip Dittmann}
\address{Charles University, Faculty of Mathematics and Physics, Department of Algebra, Sokolov\-sk\' a 83, 18600 Praha~8, Czech Republic.}
\curraddr{Université de Mons, Département de Mathematique, Place du Parc 20, 7000 Mons, Belgium}
\email{nicolas.daans@kuleuven.be}
\address{Institut für Algebra, Technische Universität Dresden, 01062 Dresden, Germany}
\curraddr{Department of Mathematics, University of Manchester, Manchester M13 9PL, United Kingdom}
\email{philip.dittmann@manchester.ac.uk}

\begin{document}
\begin{abstract}
  We establish that all rings of $S$-integers are universally definable
  in function fields in one variable over certain ground fields including global and non-archimedean local fields.
  That is, we show that the complement of such a ring of $S$-integers is always a diophantine set.
  As a technical tool,
  we use a reciprocity exact sequence for quadratic Witt groups in function fields over almost arbitrary base fields
  (of any characteristic),
  which is new and of potentially independent interest.\footnotemark[0]
\end{abstract}
\maketitle
  \footnotetext[0]{This is the Accepted Manuscript of an article published online on April 28, 2026, in \emph{Journal f{\"u}r die reine und angewandte Mathematik}.
  The published version may be found at \href{https://doi.org/10.1515/crelle-2026-0026}{https://doi.org/10.1515/crelle-2026-0026}.}

\section{Introduction}

Given a field $F$, a subset $A$ of $F$ is \emph{diophantine} or \emph{existentially definable} if there exists a polynomial $f \in F[X, Y_1, \ldots, Y_m]$, for some $m \geq 0$, such that
\begin{displaymath}
A = \lbrace x \in F \mid \exists y_1, \ldots, y_m \in F : f(x, y_1, \ldots, y_m) = 0 \rbrace.
\end{displaymath}
Unless the arithmetic of $F$ is very well-understood (e.g. when $F$ is algebraically closed, real closed or $p$-adically closed), it is hard to concretely understand diophantine sets,
and in particular to decide whether a given subset of $F$ is diophantine.

For example, letting $F$ be the field of rational numbers $\qq$, it has been a longstanding open question whether the ring of integers $\zz$, seen as a subset of $\qq$, is diophantine.
It was shown in \cite{Koe16} that $\zz$ is \emph{universally definable} in $\qq$, i.e.~its complement $\qq \setminus \zz$ is diophantine. The subsequent papers \cite{Par13,Eis18,DaansGlobal} generalised Koenigsmann's technique, and it was ultimately shown that any finitely generated subring of a global field
is universally definable in its fraction field.
Other recent advances in the study of diophantine subsets of global fields include \autocite{Morrison_cyclic,Dit17}.

In this article, we are concerned with fields $F$ which are function fields in one variable over a base field $K$,
i.e.\ $F/K$ is a finitely generated extension of transcendence degree $1$.
We prove:
\begin{stel*}[{{\Cref{EtoAglobalex}}}]
  Let $K$ be a global or non-archimedean local field, $F$ a function field in one variable over $K$.
  For any finite set $S$ of places of $F/K$, the ring of $S$-integers of $F$ is universally definable in $F$,
  i.e.\ its complement is diophantine.
\end{stel*}
One example of such a ring of $S$-integers is the ring $K[T]$ in the rational function field $F = K(T)$, so we show that $F \setminus K[T]$ is diophantine.
In the case of local fields, we can show more generally:
\begin{stel*}[{{\Cref{T:UniversalFunctionFieldMain}}}]
  Let $K$ be a non-archimedean local field, $F$ a regular function field in one variable over $K$.
  Any finitely generated $K$-subalgebra $R \subseteq F$ with fraction field $F$ is universally definable in $F$.
\end{stel*}
In fact, our theorems hold for much more general base fields than global and local fields,
although the precise hypotheses are somewhat difficult to state;
see therefore \Cref{sect:universalFunctionfield} for the exact statements.
Beyond global and local fields, the base fields $K$ to which our theorems apply include many complete discretely valued fields (like $\cc(\!(t)\!)$, the field of formal Laurent series with complex coefficients, as well as so-called higher local fields), and many pseudo-algebraically closed fields, including pseudo-finite fields, for which first-order definability of rings of $S$-integers was recently discussed independently in \cite{Nguyen-LocalSymbols}.
See \Cref{E:large-basefields-definability} for further discussion.
If $K$ is a finite field, then a function field in one variable $F/K$ is a global field;
our results can thus be seen as a variation of the aforementioned results of \cite{Eis18, DaansGlobal}.

Existing work on diophantine subsets of function fields has often been motivated by questions surrounding decidability of their (existential) first-order theories.
The undecidability of the existential first-order theory (for some suitable choice of admissible parameters) over function fields $F$
as occurring in our main theorems has been known for some time --
see \autocite{Mor05,EisHil10padic} for characteristic zero and \autocite{Shlapentokh_Pacific00} for positive characteristic
(see also \cite{EisShlap17} for quite general results in function fields in positive characteristic,
and \cite[Theorem II]{Andromeda-1} for some more recent cases, concerning new examples of base fields $K$).
These undecidability proofs required to start building a library of ``arithmetically significant'' diophantine subsets of the fields in question.
We mention in particular the role of diophantine valuation rings in function fields, first utilised (albeit implicitly) in \cite{DenefDiophantine}.
See also \cite[§6]{MillerShlapentokh_v2}, \cite{Andromeda-1} for more on diophantine valuation rings in function fields.
Beyond the prerequisites of these undecidability proofs, however, not much work seems to have been done regarding the diophantineity of arithmetically interesting subsets of such function fields.

The question of whether $\zz$ is diophantine in $\qq$ is motivated by the search for an analogue of the
Davis--Putnam--Robinson--Matiyasevich (`DPRM') Theorem \cite{Mat70}:
it is equivalent to the question of whether every computably enumerable subset of $\qq$ is diophantine.
See \cite[§4]{Pasten} for an analysis of this property for very general structures,
and in particular Proposition 4.15 there for the case of $\qq$.
Koenigsmann's universal definition of $\zz$ in $\qq$ yields the statement that
every computably enumerable relation on $\qq$ can be defined using an existential-universal formula
(see for instance \cite[Corollary 6.2]{Daans_Defining10});
this is weaker than asserting that computably enumerable sets are diophantine.
From our main results we obtain the following analogous statement
(relying on \cite[Theorem 7.1]{MillerShlapentokh_v2} instead of the DPRM Theorem):
\begin{stel*}[{{\Cref{ce-sets-ea-definable}}}]
  Let $K$ be a number field and $F$ a function field in one variable over $K$.
  Then every computably enumerable relation on $F$ can be defined using an existential-universal formula.
\end{stel*}

Let us now give a high-level overview of the proof of our main theorems.
Our method can be seen as further development of the one used in \cite{DaansGlobal} to define subrings of global fields,
later optimised in \cite{Daans_Defining10},
but is more technically demanding.
In \cite{DaansGlobal} (as in previous work),
quaternion algebras over global fields were the central arithmetical tool.
Specifically, it was shown essentially in \cite{Poo09} that valuation rings of global fields are uniformly existentially definable, and these existential definitions can be parametrised via the structure constants of quaternion algebras.
This was then already used in \cite{Poo09} to show that $\zz$ is universally-existentially definable in $\qq$; the corresponding result for rings of $S$-integers in the function fields we study was recently shown in \cite[Theorem III]{Andromeda-1}.

The chief new addition made in \cite{Koe16} to obtain universal definability of $\zz$ in $\qq$ is usage of the Quadratic Reciprocity Law.
In \cite{DaansGlobal}, in order to include fields of characteristic $2$, the following version of Quadratic Reciprocity was isolated instead as the key ingredient:
a quaternion algebra over a global field always ramifies over an even number of places,
and conversely any finite set of places of even cardinality occurs as the ramification set of a quaternion algebra.
In the present paper, instead of quaternion algebras, we systematically employ Pfister forms,
which we already used in our study of definability of valuation rings in \cite{Andromeda-1};
this entails a certain amount of care concerning quadratic form theory,
in particular in order to include the case of characteristic $2$ throughout,
although in various places we also point out possible shortcuts for the case of characteristic not $2$.
We develop an appropriate analogue for the classical results concerning ramification of quaternion algebras over global fields:
a certain reciprocity exact sequence for quadratic Witt groups in function fields replaces the role of Quadratic Reciprocity,
see \Cref{T:Reciprocity-Main}.
This appears to be new and potentially of independent interest from a quadratic form theory perspective.

The use of Pfister forms for definability purposes is by now well-established --
see for instance \cite{Pop_ElemEquivVsIso,DittmannPop_v1}, as well as \cite{Andromeda-1}.
However, previous work has mainly relied on local-global principles for Pfister forms.
Our use of the reciprocity exact sequence therefore introduces a new arithmetic tool into the study of definability.

This paper is structured as follows.
The first two sections contain no material on definability and are of a purely algebraic nature.
In particular, \Cref{sec:QF-preliminaries} discusses some background from quadratic form theory and valuation theory, which is mostly well-known, although we take extra care to formulate all statements without restrictions on the characteristic.
In \Cref{sec:reciprocity} we then state and prove the announced reciprocity exact sequence for quadratic Witt groups in function fields.
For the remainder of the paper, the only thing the reader needs to remember from \Cref{sec:reciprocity} are the statements of \Cref{T:Reciprocity-Main} and \Cref{lem:transfer-surj-kernel}.

In \Cref{sect:UniversalDefinitionGeneralTechnique} we develop in general terms the technique to obtain universal definability of rings of $S$-integers by making use of exact reciprocity laws.
In \Cref{sect:universalFunctionfield} we then put all this together to obtain our main results on definability of rings of $S$-integers in function fields, and discuss some applications.

\subsection*{Acknowledgements}
Part of this work was developed during N.~D.'s PhD project and contained in his PhD thesis \cite{DaansThesis}, prepared under supervision of Karim Becher and P.~D.\ at the University of Antwerp, and supported by the FWO PhD Fellowship fundamental research grants 51581 and 83494.
N.~D.~further gratefully acknowledges support by {Czech Science Foundation} (GA\v CR) grant 21-00420M, and {Charles University} PRIMUS Research Programme PRIMUS/24/SCI/010.
P.~D.\ would like to thank the Hausdorff Research Institute for Mathematics Bonn,
funded by the DFG under Germany's Excellence Strategy (EXC-2047/1 -- 390685813),
for its hospitality during the trimester programme ``Definability, decidability, and computability''
while this paper was revised.

We are grateful to Yong Hu for suggesting a version of the reciprocity exact sequence
(\Cref{T:Reciprocity-Main}, in a different formulation) in characteristic different from $2$,
as well as providing us with a detailed proof sketch thereof,
although the argument we ultimately give below is entirely different.

We would further like to thank Detlev Hoffmann for making us aware of the article \cite{Aravire-Jacob}.

Finally, we thank the two anonymous referees for their thorough reading of our manuscript and their invaluable suggestions, which have helped us to improve several aspects of the paper's presentation.

\section{Quadratic forms, valuations, residues, transfers}\label{sec:QF-preliminaries}
We will discuss some basic (and some less basic) concepts from quadratic form theory over fields of arbitrary characteristic.
We follow for the most part the setup of \cite{ElmanKarpenkoMerkurjev}, and refer the reader to that book for more details and context.

Let always $K$ be a field.
We denote by $\nat$ the set of natural numbers including zero, and by $\nat^+$ the set of non-zero natural numbers.

Let $V$ be a finite-dimensional vector space over $K$.
A \emph{bilinear form on $V$} is simply a bilinear map $B: V \times V \to K$; it is called \emph{symmetric} if $B(v,w) = B(w, v)$ for all $v, w \in V$.
A \emph{quadratic form on $V$} is a map $q : V \to K$ such that $q(av) = a^2q(v)$ for all $v \in V$ and $a \in K$, and such that the associated map
$$ \mf{b}_q : V \times V \to K : (v, w) \mapsto q(v+w)- q(v) - q(w),$$
called the \emph{polar form} of $q$,
is a bilinear form.
A quadratic form $q$ on some finite-dimensional vector space over $K$ will also be called a \emph{quadratic form over $K$}, the dimension of the underlying vector space $V$ will also be called the \emph{dimension} of the quadratic form $q$ and denoted by $\dim(q)$.

If $q_1$ and $q_2$ are quadratic forms on $V_1$ and $V_2$ respectively, and if there exists an embedding $L: V_1 \to V_2$ such that $q_2(L(v)) = q_1(v)$ for all $v \in V_1$, then we call $q_1$ a \emph{subform} of $q_2$.
If additionally $\dim(V_1) = \dim(V_2)$, we say that $q_1$ and $q_2$ are \emph{isometric} and denote this by $q_1 \cong q_2$.

If $f \in K[X_1, \ldots, X_n]$ is a homogeneous polynomial of degree $2$ (or $f$ is identically zero), then the map $K^n \to K : (x_1, \ldots, x_n) \mapsto f(x_1, \ldots, x_n)$ is a quadratic form.
Conversely, up to isometry, every quadratic form over $K$ is obtained in such a way.
In most contexts we will freely identify isometric quadratic forms, and as such there is no harm in representing quadratic forms over $K$ by homogeneous degree $2$ polynomials over $K$.

We call a quadratic form $q$ on $V$ \emph{isotropic} if there exists $v \in V \setminus \lbrace 0 \rbrace$ with $q(v) = 0$, \emph{anisotropic} otherwise.
We call $q$ \emph{non-singular} if the bilinear form $\mf{b}_q$ is non-degenerate, i.e.~for all $v \in V$, $v \neq 0$, there exists $w \in V$ with $q(v+w) \neq q(v) + q(w)$.
A non-singular quadratic form $q : V \to K$ is called \emph{hyperbolic} if there is a subspace $W \subseteq V$ with $2\dim(W) = \dim(V)$ and such that $q|_W$ is identically zero; by \cite[Proposition 7.28]{ElmanKarpenkoMerkurjev} a quadratic form $q$ of dimension $2n$ is hyperbolic if and only if it is isometric to the quadratic form
$$ K^{2n} \to K : (x_1, \ldots, x_n, y_1, \ldots, y_n) \mapsto \sum_{i=1}^n x_iy_i.$$

One may naturally endow the set of isometry classes of even-dimensional non-singular quadratic forms over a field $K$ with the structure of an abelian monoid.
The quotient of this monoid modulo the submonoid of hyperbolic forms then carries the structure of an abelian group, called the \emph{quadratic Witt group of $K$}, which we denote by $I_qK$.
See \cite[Section 8]{ElmanKarpenkoMerkurjev} for a precise definition and discussion of the quadratic Witt group.\footnote{The definition on \cite[51]{ElmanKarpenkoMerkurjev} seems to suggest that one needs to add formal inverses to equivalence classes of quadratic forms to obtain a group, but this is not the case: for a non-singular even-dimensional quadratic form $q : V \to K$, a representative of its inverse in $I_qK$ is given by the form $-q : V \to K : x \mapsto -q(x)$.}
Two non-singular even-dimensional quadratic forms representing the same class in $I_q K$ are called \emph{Witt equivalent}.
Every Witt equivalence class contains precisely one isometry class of an anisotropic quadratic form; as such, we may identify $I_qK$ as a set with the set of isometry classes of anisotropic non-singular even-dimensional quadratic forms over $K$.
More specifically, when $q : V \to K$ is an anisotropic non-singular even-dimensional quadratic form over $K$, the forms Witt equivalent to $q$ are precisely those isometric to
$$ V \times K^{2n} : (z, (x_1, \ldots, x_n, y_1, \ldots, y_n)) \mapsto q(z) + \sum_{i=1}^n x_iy_i$$
for some $n \in \nat$ \cite[Theorem 8.5]{ElmanKarpenkoMerkurjev}.
In particular, the neutral element of $I_q K$ is the class of all hyperbolic forms, in particular of the unique (up to isometry) zero-dimensional quadratic form over $K$, which is anisotropic and non-singular.

Similarly, on the class of non-degenerate symmetric bilinear forms, one may define the notions of isometry, isotropy, and Witt equivalence, and the Witt equivalence classes of non-degenerate bilinear forms over $K$ form a ring called the \emph{Witt ring} and denoted $WK$; see \cite[Section 1]{ElmanKarpenkoMerkurjev} for details.
We denote by $IK$ the ideal of $WK$ of classes of even-dimensional non-degenerate symmetric bilinear forms, and for $n \in \nat^+$, $I^nK$ denotes the $n$th power of the ideal $IK$.
When $\charac(K) \neq 2$, there is a natural isomorphism $I_q K \to IK$ defined by mapping the class of a quadratic form $q$ to that of its polar form $\mf{b}_q$.
For an even-dimensional non-singular quadratic form $q$, we shall denote by $[q]$ its equivalence class in $I_qK$; similarly, for a non-degenerate symmetric bilinear form $B$, we denote by $[B]$ its equivalence class in $WK$.

A special class of symmetric bilinear/quadratic forms is formed by Pfister forms.
Let us first define recursively the notion of an \emph{$n$-fold bilinear Pfister form} for $n \in \nat$.
Given a field $K$ and $a_1, \ldots, a_n \in K^\times$, the \emph{$n$-fold bilinear Pfister form} $\llangle a_1, \ldots, a_n \rrangle^b_K$ is a $K$-bilinear form on $K^{2^n}$.
The $0$-fold bilinear Pfister form $\llangle\rrangle^b_K$ is the bilinear form on the $1$-dimensional vector space $K$ given by
$$ K \times K \to K : (x, y) \mapsto xy. $$
Given an $n$-fold bilinear Pfister form $B = \llangle a_1, \ldots, a_n \rrangle^b_K : V \times V \to K$ and $a_{n+1} \in K^\times$, $\llangle a_1, \ldots, a_{n+1} \rrangle^b_K$ is defined by the mapping
$$ (V \times V) \times (V \times V) \to K : ((v_1 , v_2), (w_1 , w_2)) \mapsto B(v_1, w_1) - a_{n+1}B(v_2, w_2). $$
We refer to \autocite[Section 6]{ElmanKarpenkoMerkurjev} for more details.

For later use, we mention a result on bilinear forms over valued fields.
Our notational conventions concerning valuations are as follows.
For a (Krull) valuation $v$ on a field $K$, we denote by $\mc{O}_v$, $\mf{m}_v$, $vK$, $Kv$, $K_v$ the valuation ring, valuation ideal, value group, residue field, and henselisation, respectively.
When $x \in \mc{O}_v$ we shall denote by $\ovl{x}$ its equivalence class in $Kv = \mc{O}_v/\mf{m}_v$, assuming that the valuation is clear from the context.
We call a valuation $v$ \emph{dyadic} if $\charac(Kv) = 2$ (or equivalently, if $v(2) > 0$), \emph{non-dyadic} otherwise.
We will use the term \emph{$\zz$-valuation} for a valuation whose value group is precisely $\zz$.

The argument for the following borrows ideas from the proof of \autocite[Proposition 4.1]{BecherRaczek}.
\begin{lem}\label{L:Pfisterbilinear}
Let $B = \llangle x_1, \ldots, x_n \rrangle_K^b$ for some $n \in \mathbb{N}^+$ and $x_1, \ldots, x_n \in K^\times$, and let $S$ be a finite set of $\zz$-valuations on $K$. Then there exist $a_1, \ldots, a_n \in K^\times$ such that $B \cong \llangle a_1, \ldots, a_n \rrangle_K^b$ and $a_1, \ldots, a_{n-1} \in \bigcap_{v \in S} \mathcal{O}_v^\times$.
\end{lem}
\begin{proof}
We may assume that $S \neq \emptyset$, since otherwise we can set $a_i = x_i$ for $1 \leq i \leq n$.
For $n=1$ there is nothing to show, we now consider the case $n=2$.

Write $B = \llangle x, y \rrangle^b_K$.
By the Weak Approximation Theorem, we may multiply $x$ and $y$ by an appropriate square to obtain without loss of generality that for all $v \in S$ one has  $(v(x), v(y)) \in \lbrace (0, 1), (1, 0), (1, 1), (2, 2) \rbrace$ and $x \neq -y$.
Using Weak Approximation again, we can then find $r, s, t \in \bigcap_{v \in S} \mathcal{O}_v$ such that $x = ts$, $y = rs$, $ts + rs \neq tr$,  and for every $v \in S$, precisely one of $r, s, t$ lies in $\mathfrak{m}_v$.

Now consider the following computation, invoking repeatedly \autocite[Lemma 4.15~(2)]{ElmanKarpenkoMerkurjev}:
\begin{displaymath}
\llangle x, y \rrangle^b_K \cong \llangle x+y, -xy \rrangle^b_K \cong \llangle ts + rs, -tr \rrangle^b_K \cong \llangle ts + rs - tr, tr(ts + rs) \rrangle^b_K.
\end{displaymath}
Setting $a_1 = ts + rs - tr$ and $a_2 = tr(ts + rs)$, it follows that $B \cong \llangle a_1, a_2 \rrangle^b_K$. By the assumptions on $r, s, t$ we also obtain that $a_1 \in \bigcap_{v \in S} \mathcal{O}_v^\times$.
This concludes the proof in the case $n=2$.

For $n \geq 3$ we argue by induction on $n$. By the induction hypothesis, we may first find $a_{n-1}' \in K^\times$ and $a_1, \ldots, a_{n-2} \in \bigcap_{v \in S} \mc{O}_v^\times$ such that $\llangle x_1, \ldots, x_{n-1} \rrangle_K^b \cong \llangle a_1, \ldots, a_{n-2}, a_{n-1}' \rrangle_K^b$, and then we may find $a_n \in K^\times$ and $a_{n-1} \in \bigcap_{v \in S} \mc{O}_v^\times$ such that $\llangle a_{n-1}', x_n \rrangle_K^b \cong \llangle a_{n-1}, a_n \rrangle_K^b$.
One then easily computes that
$$
B \cong \llangle a_1, \ldots, a_{n-2}, a_{n-1}', x_n \rrangle_K^b \cong \llangle a_1, \ldots, a_{n-2}, a_{n-1}, a_n \rrangle_K^b
$$
as desired (using that we have the computation rule
$$ \llangle y_1, \ldots, y_i, y_{i+1}, \ldots, y_n \rrangle_K^b \cong \llangle y_1, \ldots, y_{i-1}, y_{i+1}, y_i, y_{i+2}, \ldots, y_n \rrangle_K^b$$
for all $1 \leq i \leq n-1$ and $y_1, \ldots, y_n \in K^\times$.)
\end{proof}

Let us now define the notion of an \emph{$(n+1)$-fold quadratic Pfister form} for $n \in \nat$.
Let $a_1, \ldots, a_{n}, b \in K$ with $a_1, \ldots, a_{n}, 1+4b \neq 0$.
The $1$-fold quadratic Pfister form $\llangle b ]]_K$ is the form
$$  \llangle b ]]_K : K^2 \to K : (x_1, x_2) \mapsto x_1^2 - x_1 x_2 - bx_2^2. $$
Given an $n$-fold quadratic Pfister form $q = \llangle a_1, \ldots, a_{n-1}, b]]_K : V \to K$ and $a_n \in K^\times$, we define the $(n+1)$-fold quadratic Pfister form $\llangle a_1, \ldots, a_{n}, b]]_K$ as
$$ \llangle a_1, \ldots, a_{n}, b]]_K : V^2 \to K : (v_1, v_2) \mapsto q(v_1) - a_{n}q(v_2).$$
Equivalently, one sees that $\llangle a_1, \ldots, a_n, b]]_K$ is the \emph{tensor product} of the bilinear form $\llangle a_1, \ldots, a_n \rrangle^b_K$ and the $1$-fold Pfister form $\llangle b]]_K$, i.e.~$\llangle a_1, \ldots, a_n, b]]_K = \llangle a_1, \ldots, a_n \rrangle_K^b \otimes \llangle b ]]_K$; see \cite[51]{ElmanKarpenkoMerkurjev} for a general definition of the tensor product of a bilinear and a quadratic form.
In the rest of this paper, when we refer to Pfister forms, we will always mean quadratic Pfister forms.\footnote{Note
  that as in \cite{Andromeda-1},
  our notation $\llangle a_1, \dotsc, a_n, b]]_K$ disagrees with the notation in \cite{ElmanKarpenkoMerkurjev}, but leads to the same notion of quadratic Pfister forms (up to isometry);
  see \cite[Remark 3.2]{Andromeda-1}.}

An $n$-fold Pfister form over $K$ is non-singular, and is either anisotropic or corresponds to the zero element in the group $I_qK$.
For $n \geq 2$ denote by $I_q^n K$ the subgroup of $I_qK$ generated by the classes of $n$-fold Pfister forms, and set $I_q^1K = I_qK$.
We then have $I_qK = I_q^1K \supseteq I_q^2 K \supseteq I_q^3K \supseteq \ldots$.
If $\charac(K) \neq 2$, then under the aforementioned isomorphism $I_qK \to IK$, $I^n_q K$ maps onto $I^n K$.

We now discuss some maps between quadratic Witt groups over different fields.
Let us first consider a field extension $L/K$.
There is a natural \emph{restriction} homomorphism $I_q K \to I_q L$ mapping the class of a quadratic form $q : V \to K$ to that of its scalar extension to $L$, i.e.~the unique quadratic form $q_L$ on the $L$-vector space $V \otimes_K L$ such that $q_L(v \otimes b) = b^2q(v)$ and $\mf{b}_{q_L}(v \otimes b, w \otimes c) = bc\mf{b}_{q}(v, w)$ for $v, w \in V$ and $b, c \in L$.
This restriction homomorphism maps the class of an $n$-fold Pfister form to the class of an $n$-fold Pfister form, and thus in particular maps $I^n_q K$ to $I^n_q L$ for all $n \in \nat^+$.

When $L/K$ is a finite extension and $s : L \to K$ a non-zero $K$-linear map, then we can define a group homomorphism $s^\ast : I_q L \to I_q K$ mapping the class of a quadratic form $q$ over $L$ to class of the quadratic form $s \circ q$ over $K$.
Such a map $s^\ast$ is called a \emph{(Scharlau) transfer}.
See \cite[Section 10]{ElmanKarpenkoMerkurjev} for details.
In general, the map $s^\ast$ depends on the choice of the linear map $s$.
The following lemma will be used in the next section to isolate certain cases where a part of the map $s^\ast$ does not depend on the choice of $s$.
(Compare with the similar result \cite[Exercise 21.6]{ElmanKarpenkoMerkurjev} for bilinear forms.)
\begin{lem}\label{lem:transfer-unique}
Let $L/K$ be a finite field extension, $d > 0$ such that $I_q^{d+1} L = 0$.
For any two non-zero $K$-linear maps $s, t : L \to K$, the induced transfer maps $s^\ast$ and $t^\ast$ satisfy $s^\ast\vert_{I_q^d L} = t^\ast\vert_{I_q^d L}$.
\end{lem}
\begin{proof}
Since the $K$-linear map $L \to \Hom_K(L, K) : a \mapsto (x \mapsto s(ax))$ is injective, it must also be surjective by comparing dimensions.
Hence there exists $a \in L$ such that $t^\ast(\alpha) = s^\ast(a\alpha)$ for all $\alpha \in I_q L$.

Now we compute that for $\alpha \in I_q^d L$ arbitrary,
$$ s^\ast(\alpha) - t^\ast(\alpha) = s^\ast(\alpha) - s^\ast(a\alpha) = s^\ast(\alpha - a\alpha) = s^\ast(0) = 0 $$
where we used that $\alpha - a\alpha \in I_q^{d+1} L = \{ 0 \}$.
We conclude that indeed $t^\ast(\alpha) = s^\ast(\alpha)$ for all $\alpha \in I_q^d L$.
\end{proof}

We now consider quadratic forms over valued fields.
We wish to relate quadratic forms over a valued field to forms over its residue field.
The theory here is very well-known in the non-dyadic case, see \Cref{rem:residue-quad} below;
however, we cover both dyadic and non-dyadic valuations in a way suitable for our applications,
following a combination of the approaches of \cite{Tietze,MMW91,SpringerTameQuadratic}.

Let $(K, v)$ be a henselian valued field.
Let $q : V \to K$ be an anisotropic quadratic form defined over $K$.
Define the following subsets of $V$:
\begin{displaymath}
V_0 = \lbrace x \in V \mid v(q(x)) \geq 0 \rbrace, \qquad V_0^\circ = \lbrace x \in V \mid v(q(x)) > 0 \rbrace.
\end{displaymath}
$V_0$ and $V_0^\circ$ are $\mc{O}_v$-submodules of $V$ (by the Schwarz Inequality, see e.g.~\cite[Lemma 9]{SpringerTameQuadratic}) and $V_0/V_{0}^\circ$ is naturally a $Kv$-vector space.
We can define an anisotropic quadratic form
$$\RF{1}{q} : V_0/V_{0}^\circ \to Kv : \ovl{x} \mapsto \ovl{q(x)}$$
where $\ovl{x}$ denotes the equivalence class of a vector $x \in V_0$.
We call this the \emph{first residue form of $q$}.
Even when $q$ is itself non-singular, $\RF{1}{q}$ might be singular
(although this only occurs in characteristic $2$).

We can now define the residue homomorphism with respect to a $\zz$-valuation of a field $K$,
which is defined on the graded components of the Witt group $I_q K$, i.e.\ the groups $I_q^n K / I_q^{n+1} K$.
Slightly abusing notation, we will let maps defined on a quotient like $I_q^n K / I_q^{n+1} K$ take elements of $I_q^n K$ as input, which the reader should then interpret as taking their equivalence classes modulo $I_q^{n+1} K$ as input.
Furthermore, if $I_q^{n+1} K = 0$, we naturally identify $I_q^n K$ and $I_q^n K/I_q^{n+1} K$.
For a field $K$, denote by $K\pow{2}$ its subset of squares; this is a subfield if $\charac(K) = 2$.

\begin{prop}\label{C:residueHomomorphism}
Let $(K, v)$ be a $\zz$-valued field and $n \in \nat^+$.
If $v$ is dyadic, assume that $[Kv : (Kv)\pow{2}] < 2^{n}$.
There is a unique surjective group homomorphism
$$ \partial_v : I^{n+1}_qK/I^{n+2}_qK \to I^n_q(Kv)/I^{n+1}_q(Kv) $$
such that, if $q$ is a non-singular quadratic form over $K$ with $[q] \in I_q^{n+1}K$,
then $\partial_v[q] = [\RF{1}{q_{K_v}}] + I_q^{n+1} Kv$ if the form $q_{K_v}$ over the henselisation $K_v$ is anisotropic,
and $\partial_v[q] = 0 + I_q^{n+1} Kv$ if $[q_{K_v}] = 0$.
If additionally $q$ is an $(n+1)$-fold Pfister form, then $\partial_v[q]$ is the class of an $n$-fold Pfister form.
Furthermore, if $v$ is henselian and $I^{n+1}_q(Kv) = 0$, then $\partial_v$ is an isomorphism.
\end{prop}
\begin{proof}
This is a special case of \cite[Corollary 4.10]{DaansLinkage}: the map $\partial_v$ defined there has as its domain $(I_q^{[v]}K)^{n+1}/(I_q^{[v]}K)^{n+2}$ where $(I_q^{[v]}K)^{n+1}$ and $(I_q^{[v]}K)^{n+2}$ are certain subgroups of $I_q^{n+1}K$ and $I_q^{n+2}K$ respectively, but we have $(I_q^{[v]}K)^{n+1} = I_q^{n+1}K$ and $(I_q^{[v]}K)^{n+2} = I_q^{n+2}K$ if $v$ is non-dyadic, or $v$ is dyadic and $[Kv : (Kv)\pow{2}] < 2^n$ \cite[Lemma 5.4]{DaansLinkage}, hence the desired statement follows.
\end{proof}
The maps $\partial_v$ described in \Cref{C:residueHomomorphism} will be called \emph{residue homomorphisms}.
They admit a very concrete description as follows.
\begin{lem}\label{P:Pfisterresidue}
  Let $(K, v)$ be a henselian $\zz$-valued field, let $n \in \nat^+$, and let $a_1, \ldots, a_n, b \in K$ with $a_1, \ldots, a_n, 1+4b \neq 0$.
  Let $q = \llangle a_1, \ldots, a_n, b ]]_K$.
  Suppose that $v(a_1) = \dotsb = v(a_{n-1}) = v(b) = v(1+4b) = 0$.
\begin{enumerate}[$(1)$]
\item\label{it:PfisterResidueUnramified} If $v(a_n)=0$,
  then $q$ is anisotropic if and only if $\llangle \overline{a_1}, \ldots, \overline{a_n}, \overline{b} ]]_{Kv}$ is anisotropic,
  and in this case $\RF{1}{q} \cong \llangle \overline{a_1}, \ldots, \overline{a_n}, \overline{b} ]]_{Kv}$.
  Therefore $\partial_v[q] \in I_q^n(Kv)/I_q^{n+1}(Kv)$ vanishes.
\item\label{it:PfisterResidueRamified} If $v(a_n)$ is odd,
  then $q$ is anisotropic if and only if $\llangle \overline{a_1}, \ldots, \overline{a_{n-1}}, \overline{b}]]_{Kv}$ is anisotropic,
  and in this case $\RF{1}{q} \cong \llangle \overline{a_1}, \ldots, \overline{a_{n-1}}, \overline{b}]]_{Kv}$.
  Therefore $\partial_v[q] \in I_q^n(Kv)/I_q^{n+1}(Kv)$ is the class of the Pfister form $\llangle \overline{a_1}, \dotsc, \overline{a_{n-1}}, \overline{b}]]_{Kv}$.
\end{enumerate}
\end{lem}
\begin{proof}
This is a special case of \cite[Proposition 4.8]{DaansLinkage}.
\end{proof}

\Cref{P:Pfisterresidue} completely describes the behaviour of $\partial_v$,
since in the relevant situation every Pfister form has a presentation of the required shape.
We observe this in the following strong form.
\begin{lem}\label{P:Pfister-semilocal-presentation}
Let $K$ be a field, $n \in \nat^+$, and let $S$ be a finite set of $\zz$-valuations on $K$.
Let $q$ be an $(n+1)$-fold Pfister form and assume that for all dyadic $v \in S$ we have $[Kv : (Kv)\pow{2}] < 2^{n}$.
Then there exist $a_1, \ldots, a_{n}, b \in K$ with $v(a_1) = \ldots = v(a_{n-1}) = v(b) = v(1+4b) = 0$ and $v(a_n) \in \{ 0, 1 \}$ for all $v \in S$ and such that $q \cong \llangle a_1, \ldots, a_n, b]]_K$.
\end{lem}
\begin{proof}
By \cite[Corollary~4.16]{Andromeda-1} we may find $b \in K$ with $v(b) = v(1+4b) = 0$ for all $v \in S$ such that $q \cong \llangle a_1', \ldots, a_n', b]]_K$ for certain $a_1', \ldots, a_n' \in K^\times$.
By \Cref{L:Pfisterbilinear} we may find $a_1, \ldots, a_n \in K^\times$ with $v(a_1) = \ldots, v(a_{n-1}) = 0$ for all $v \in S$ and such that $\llangle a_1, \ldots, a_n \rrangle_K^b \cong \llangle a_1', \ldots, a_n' \rrangle_K^b$.
By Weak Approximation, we may find $x \in K^\times$ such that $v(x) = -\lfloor v(a_n)/2 \rfloor$ for all $v \in S$.
We may replace $a_n$ by $a_n x^2$ without affecting the isometry class of $\llangle a_1, \ldots, a_n \rrangle_K^b$, and thus assume without loss of generality that $v(a_n) \in \{ 0, 1 \}$.
It follows that $q \cong \llangle a_1', \ldots, a_n', b]]_K \cong \llangle a_1, \ldots, a_n, b]]_K$ and this presentation is as desired.
\end{proof}

\begin{opm}\label{rem:residue-quad}
In this remark we consider the situation of non-dyadic $\zz$-valued fields.
The notion of residue homomorphism is well-known in this case.
See for example \cite[Section 19.A]{ElmanKarpenkoMerkurjev} for a discussion of residue homomorphisms defined on symmetric bilinear forms,
and recall from \Cref{sec:QF-preliminaries} that for fields of characteristic different from $2$, there is a natural correspondence between quadratic and symmetric bilinear forms.
The presentation there, defining residue homomorphisms through an explicit presentation of the Witt group in terms of generators and relations,
is quite different from the abstract one chosen here.
It is clear from \Cref{P:Pfisterresidue} and \cite[Example 19.13]{ElmanKarpenkoMerkurjev}, however,
that for a Pfister form $q$,
taking the class in the Witt group of $Kv$ of $\RF{1}{q}$ in our notation corresponds to applying
the first residue homomorphism $\partial$ to the class of $q$ in the notation there.
There is also a second residue homomorphism with respect to a choice of uniformiser $\pi$,
which corresponds to using the residue form $\RF{\pi}{q}$ instead of $\RF{1}{q}$ in the notation of \cite{DaansLinkage}.
For the definition of the map $\partial_v \colon I^{n+1}_qK/I^{n+2}_qK \to I^n_q(Kv)/I^{n+1}_q(Kv)$,
all the residue homomorphisms yield the same result
(see for instance \cite[Satz 3.1]{Arason_CohomInvarianten} and its proof).
\end{opm}

\begin{opm}\label{rem:residue-K-cohom}
There are also related residue homomorphism in Milnor $K$-theory (see for instance \cite[Section 100.B]{ElmanKarpenkoMerkurjev})
and in Galois cohomology \cite[Construction 6.8.5]{GilleSzamuely_2nd}.
For non-dyadic $\zz$-valued fields, all of these presentations are compatible with one another
(and with the one of \Cref{C:residueHomomorphism}) in view of the natural isomorphisms
\[ K^M_n(K)/2K^M_n(K) \cong I_q^n(K)/I_q^{n+1}(K) \cong H^n(K, \zz/2\zz) \]
for an arbitrary field $K$ of characteristic not $2$ conjectured by Milnor
(now theorems of Voevodsky and Orlov--Vishik--Voevodsky, see \cite[Section 16]{ElmanKarpenkoMerkurjev}),
see eg.~\cite[Proposition 7.5.1]{GilleSzamuely_2nd} for the compatibility.
\end{opm}

For use in the proof of \Cref{lem:residue-transfer-commute} in the next section,
we mention two more technical results on computing with residue homomorphisms and transfers.
It is convenient to make use of quadratic and bilinear forms over discrete valuation rings (as opposed to over fields).
All basic concepts of quadratic and bilinear forms over fields carry over to this context, see e.g.~\cite[Section 19]{ElmanKarpenkoMerkurjev} for a brief discussion.
In the following $\mathcal{O}_v$ is a discrete valuation ring corresponding to a $\mathbb{Z}$-valuation $v$ on the fraction field $K$ of $\mathcal{O}_v$.
Given a finite-dimensional free $\mc{O}_v$-module $\mathfrak{V}$, a quadratic form on $\mathfrak{V}$ is a map $q : \mathfrak{V} \to \mc{O}_v$ such that $q(ax) = a^2q(x)$ for all $a \in \mc{O}_v$ and $x \in \mathfrak{V}$, and for which the associated map
$$ \mf{b}_q : \mathfrak{V} \times \mathfrak{V} \to \mc{O}_v : (x, y) \mapsto q(x+y) - q(x) - q(y) $$
is bilinear.
A quadratic form $q : \mathfrak{V} \to \mc{O}_v$ over $\mathcal{O}_v$ is said to be \emph{non-singular} if the linear map $\mathfrak{V} \to \Hom_{\mc{O}_v}(\mathfrak{V}, \mc{O}_v) : v \mapsto (w \mapsto b_q(v, w))$ is an isomorphism.
It is \emph{hyperbolic} if it is non-singular and there is a free submodule $\mathfrak{W} \subseteq \mathfrak{V}$ such that $2\dim(\mathfrak{W}) = \dim(\mathfrak{V})$ and $q\vert_{\mathfrak{W}}$ is identically zero.
Furthermore, we denote by $q_K$ its scalar extension to $K$, which is a quadratic form on $\mathfrak{V} \otimes_{\mc{O}_v} K$, defined similarly to scalar extensions over fields.
By $\ovl{q}$ we denote the form over $Kv$ given by $\mathfrak{V}/\mf{m}_v\mathfrak{V} \to Kv : x + \mf{m}_v\mathfrak{V} \mapsto \ovl{q(x)}$.

In particular, if $V$ is a finite-dimensional $K$-vector space on an arbitrary field $K$, $q : V \to K$ a quadratic form, $v$ a $\zz$-valuation on $K$ and $\mathfrak{V}$ a free $\mc{O}_v$-submodule of $V$ such that $q(\mathfrak{V}) \subseteq \mc{O}_v$, then $q\vert_{\mathfrak{V}}$ is a quadratic form over $\mc{O}_v$, and if additionally $\lbrace bv \mid b \in K, v \in \mathfrak{V} \rbrace = V$, then $(q\vert_{\mathfrak{V}})_K \cong q$.

Let $(K, v)$ be a henselian $\mathbb{Z}$-valued field.
A quadratic form over $K$ which is of the form $q_K$ for some non-singular quadratic form $q : \mathfrak{V} \to \mc{O}_v$ over $\mc{O}_v$ is called \emph{unimodular}.
If $q$ is additionally anisotropic, then one sees that $(\mathfrak{V} \otimes_{\mc{O}_v} K)_0 = \mathfrak{V}$ and $(\mathfrak{V} \otimes_{\mc{O}_v} K)_0^\circ = \mf{m}_v\mathfrak{V}$, whereby $\RF{1}{q_K} \cong \overline{q}$, which is also non-singular.
Conversely, if $q : V \to K$ is an anisotropic quadratic form such that $\RF{1}{q}$ is non-singular and $\dim(q) = \dim(\RF{1}{q})$, then $\mathfrak{V} = V_0 = \lbrace x \in V \mid v(q(x)) \geq 0 \rbrace$ is a free $\mc{O}_v$-submodule of $V$ \cite[Proposition 5(ii)]{MMW91}, and furthermore $\mf{m}_v \mathfrak{V} = V_0^\circ$.
One sees that $q\vert_{\mathfrak{V}}$ is a non-singular quadratic form over $\mc{O}_v$ and $q \cong (q\vert_{\mathfrak{V}})_K$.
We conclude that an anisotropic form $q$ over $K$ is unimodular if and only if $\RF{1}{q}$ is non-singular and $\dim(q) = \dim(\RF{1}{q})$.

If $q_1$ and $q_2$ are both anisotropic unimodular forms over $K$, then they are isometric if and only if their first residue forms $\RF{1}{q_1}$ and $\RF{1}{q_2}$ are isometric.
Furthermore, the classes of even-dimensional unimodular forms naturally form a subgroup of $I_q K$, and if $q_1$, $q_2$, and $q_3$ are anisotropic unimodular forms such that $[q_1] + [q_2] = [q_3]$, then $[\RF{1}{q_1}] + [\RF{1}{q_2}] = [\RF{1}{q_3}]$.
See e.g.~\cite[Satz 3.1 \& 4.1]{Tietze} for more general statements and their proofs.

The following is an addendum to the computation rules for residue forms developed earlier.

\begin{lem}\label{C:Pfisterresidue}
Let $(K, v)$ be a henselian $\zz$-valued field with uniformiser $\pi$.
Let $n \in \nat^+$, assume that $I_q^{n+1}(Kv) = 0$, and if $v$ is dyadic, assume that $[Kv : (Kv)\pow{2}] < 2^n$.
Then every class of $I_q^{n+1} K$ is of the form $[q'] - [\pi q']$ for an anisotropic unimodular form $q'$, which is unique up to isometry. Furthermore, we then have $\partial_v([q'] - [\pi q']) = [\RF{1}{q'}]$.
\end{lem}
\begin{proof}
We first establish the uniqueness: suppose that $[q'] - [\pi q'] = [q''] - [\pi q'']$ for anisotropic unimodular forms $q'$ and $q''$.
Then $[q'] - [q'']$ is the class of an anisotropic unimodular form $\tilde{q}$ which satisfies $[\tilde{q}] = [q'] - [q''] = [\pi q'] - [\pi q''] = [\pi \tilde{q}]$.
Since an anisotropic unimodular form cannot represent a uniformiser, $[\tilde{q}] = [\pi \tilde{q}]$ forces $[\tilde{q}] = 0$, implying that $[q'] = [q'']$ and hence $q' \cong q''$.

For the existence statement, first consider an $(n+1)$-fold Pfister form $q$ over $K$ such that $[q] \neq 0$ (equivalently, by \Cref{C:residueHomomorphism}, $\partial_v[q] \neq 0$).
By \Cref{P:Pfister-semilocal-presentation} and \Cref{P:Pfisterresidue} we have $q \cong \llangle a_1, \ldots, a_n, b]]_K$ for some $a_1, \ldots, a_n, b \in K$ with $v(a_1) = \ldots = v(a_{n-1}) = v(b) = v(1+4b) = 0$ and $v(a_n) = 1$.
Write $q' = \llangle a_1, \ldots, a_{n-1}, b]]_K$
and observe that
$$ \partial_v([q'] - [a_n\pi^{-1}q']) = \partial_v(\llangle a_1, \ldots, a_{n-1}, a_n\pi^{-1}, b]]_K) = 0 $$
by \Cref{P:Pfisterresidue}, so that $[q'] = [a_n \pi^{-1} q']$ in view of \Cref{C:residueHomomorphism}, and consequently $[a_n q'] = [\pi q']$.
We obtain that $[q] = [q'] - [a_nq'] = [q'] - [\pi q'] = [\llangle a_1, \ldots, a_{n-1}, \pi, b ]]_K]$.
Using \Cref{P:Pfisterresidue} we now compute that
$$ \partial_v([q'] - [\pi q']) = \partial_v([\llangle a_1, \ldots, a_{n-1}, \pi, b ]]_K]) = [\llangle \overline{a_1}, \ldots, \overline{a_{n-1}}, \overline b ]]_K] = [\RF{1}{q'}]$$
as desired.

By definition of $I_q^{n+1} K$, every class of $I_q^{n+1} K$ is a finite sum of classes of $(n+1)$-fold Pfister forms.
The desired statement thus follows in general, since sums of classes of unimodular forms are again classes of unimodular forms, $\partial_v$ is a group homomorphism, and $[\RF{1}{q_1}] + [\RF{1}{q_2}] = [\RF{1}{q_3}]$ for anisotropic unimodular forms $q_1, q_2, q_3$ with $[q_1] + [q_2] = [q_3]$.
\end{proof}

In the non-dyadic case, a related version of the following statement
(which would be sufficient for our purposes)
is observed in \cite[Satz 2.5]{Arason_CohomInvarianten}.
\begin{lem}\label{L:transfer-unimodular}
Let $(K, v)$ be a henselian $\zz$-valued field, $(L,w)/(K,v)$ a finite extension of valued fields.
Consider a $K$-linear map $s \colon L \to K$ with $s(\mc{O}_w) \subseteq \mc{O}_v$ and $s(\mf{m}_w) \subseteq \mf{m}_v$.
Let $\ovl{s}$ be the induced $Kv$-linear map $Lw \to Kv : \ovl{x} \mapsto \ovl{s(x)}$.
Assume that $\ovl{s}$ is not the zero map.
Let $q$ be an anisotropic, unimodular, even-dimensional quadratic form over $L$.
Denote by $q'$ the unique anisotropic quadratic form over $K$ such that $[q'] = s^\ast([q])$.

Then $q'$ is unimodular, and $\ovl{s}^\ast([\RF{1}{q}]) = [\RF{1}{q'}]$.
\end{lem}

\begin{proof}
Write $q : V \to L$ for some $L$-vector space $V$, and as before, set $\mathfrak{V} = V_0 = \lbrace x \in V \mid w(q(x)) \geq 0 \rbrace$.
As before, since $q$ is anisotropic, this is an $\mc{O}_w$-submodule of $V$, and since $q$ is unimodular, $\mathfrak{V}$ is a free $\mc{O}_w$-module.
Further recall that $\mf{m}_w \mathfrak{V} = V_0^\circ = \lbrace x \in V \mid w(q(x)) > 0 \rbrace$, and for $x \in \mathfrak{V}$, we denote by $\ovl{x}$ its equivalence class in $\mathfrak{V}/\mf{m}_w \mathfrak{V}$.

We compute that, for $x \in \mathfrak{V}$, $(\ovl{s} \circ \RF{1}{q})(\ovl{x}) = \ovl{s}(\RF{1}{q}(\ovl{x})) = \ovl{s}(\ovl{q(x)}) = \ovl{s(q(x))} = \ovl{s \circ q\vert_{\mathfrak{V}}}(\ovl{x})$.
As $q$ is anisotropic unimodular, $\RF{1}{q}$ is non-singular, and then by \cite[Lemma 20.4]{ElmanKarpenkoMerkurjev} we have that $\ovl{s} \circ \RF{1}{q}$ is non-singular.
It follows that also $s \circ q\vert_{\mathfrak{V}}$, as a quadratic form over $\mc{O}_v$, is non-singular, and $\overline{s \circ q\vert_{\mathfrak V}} \cong \overline{s} \circ \RF{1}{q}$.

By \cite[Corollary 19.3]{ElmanKarpenkoMerkurjev} we may find free $\mc{O}_v$-submodules $\mathfrak{W}, \mathfrak{W}'$ of $\mathfrak{V}$ such that $\mathfrak{V} = \mathfrak{W} \oplus \mathfrak{W}'$, $\mathfrak{W}$ and $\mathfrak{W}'$ are orthogonal with respect to $s \circ q\vert_{\mathfrak V}$, $s \circ q\vert_{\mathfrak{W}}$ (or, equivalently, $(s \circ q\vert_{\mathfrak{W}})_K$) is anisotropic, and $s \circ q\vert_{\mathfrak{W}'}$ is hyperbolic.
%In particular, $s \circ q \vert_{\mathfrak W}$ is non-singular and $(s \circ q\vert_{\mathfrak{W}})_K \cong q'$, whereby $q$ is unimodular with $\RF{1}{q'} \cong \overline{s \circ q\vert_{\mathfrak W}}$.
Writing $\ovl{\mathfrak{W}} = \lbrace \ovl{x} \mid x \in \mathfrak{W} \rbrace$ and $\ovl{\mathfrak{W}'} = \lbrace \ovl{x} \mid x \in \mathfrak{W}' \rbrace$ and passing to the residue field, 
it follows that $\ovl{\mathfrak{W}}$ and $\ovl{\mathfrak{W}'}$ are orthogonal with respect to $\ovl{s} \circ \RF{1}{q}$, that $\ovl{s} \circ \RF{1}{q}\lvert_{\ovl{\mathfrak{W}}}$ is anisotropic (since $s \circ q\vert_{\mathfrak W}$ is anisotropic non-singular),
and that $\ovl{s} \circ \RF{1}{q}\vert_{\ovl{\mathfrak{W}'}}$ is hyperbolic (as it is the reduction modulo $\mf{m}_v$ of the hyperbolic form $s \circ q\vert_{\mathfrak{W}'}$).
We obtain that $q' \cong (s \circ q\vert_{\mathfrak W})_K$ is unimodular and $\RF{1}{q'} \cong \overline{s \circ q\vert_{\mathfrak W}} \cong \ovl{s} \circ \RF{1}{q}\vert_{\ovl{\mathfrak{W}}}$, which implies the desired statement.
\end{proof}

\section{Reciprocity}
\label{sec:reciprocity}

In this section, we prove a reciprocity law for Pfister forms over function fields, \Cref{T:Reciprocity-Main}.
To state our result, we first need to consider an appropriate notion of dimension for (function) fields.
For a field $K$, we denote by $\cd_2(K)$ the \emph{$2$-cohomological dimension} of $K$.
We refer to \cite[Section VI.5]{Neu15} for a definition.
Except for in the proof of \Cref{lem:transfer-surj-kernel}, we will only require the following characterisation of $2$-cohomological dimension
deduced from well-known facts, which the reader may thus take as a definition.
\begin{prop}\label{P:cdCharacterisation}
Let $K$ be a field with $\charac(K) \neq 2$ and $n \in \nat$.
Then $\cd_2(K) \leq n$ if and only if, for every finite separable field extension $L/K$, $I_q^{n+1} L = 0$.
\end{prop}
\begin{proof}
It is a known characterisation of $p$-cohomological dimension for a prime number $p \neq \charac(K)$ that $\cd_p(K) \leq n$ if and only if, for every finite separable extension $L/K$, $H^{n+1}(L, \mu_p) = 0$, where $\mu_p$ is the group of $p$-th roots of unity with the usual action of the absolute Galois group of $L$; see for example the proof of \cite[Proposition 6.5.11]{Neu15}.
For the case $p=2$, by the Milnor Conjectures (Theorem of Orlov--Vishik--Voevodsky, see e.g.~\cite[Fact 16.2]{ElmanKarpenkoMerkurjev})
$H^{n+1}(L, \mu_2)$ is isomorphic as a group to $I^{n+1}_q L/I_q^{n+2} L$.
(In characteristic away from $2$, \cite{ElmanKarpenkoMerkurjev} write $H^{n+1}(L)$ for the Galois cohomology group $H^{n+1}(L, \mu_2)$,
see the definition in \cite[§101, p.~402 and p.~404]{ElmanKarpenkoMerkurjev},
noting that the Galois module $\mu_2^{\otimes m}$ is isomorphic $\mathbb{Z}/2\mathbb{Z}$ for all $m$
since the absolute Galois group acts trivially on $\mu_2 = \{ \pm 1 \}$.)
Finally, we have $I^{n+1}_q L/I_q^{n+2} L = 0$ if and only if $I^{n+1}_q L = I^{n+2}_q L$ if and only if $I_q^{n+1} L = 0$:
indeed, the group $I^{n+1}_q L$ is by definition generated by the classes of $(n+1)$-fold Pfister forms,
but by the Arason-Pfister Hauptsatz \cite[Theorem 23.7~(1)]{ElmanKarpenkoMerkurjev}
an $(n+1)$-fold Pfister form whose class lies in $I^{n+2}_q L$ must be isotropic and
hence its class is the zero element in $I^{n+1}_q L$,
so $I^{n+1}_q L = I^{n+2}_q L$ is only possible if both groups are trivial.
\end{proof}

For an arbitrary field $K$, we define the \emph{$2$-dimension} $\dim_2(K) \in \mathbb{N} \cup \{ \infty \}$
to be $\cd_2(K)$ if $\charac(K) \neq 2$,
and otherwise the quantity $d$ with $2^d = 2[K : K\pow{2}]$
(in other words, $\dim_2 K$ is $1$ plus the imperfect exponent of $K$ in the second case).

The case distinction in the definition is necessitated by the bad behaviour of $2$-cohomological dimension in characteristic $2$,
where $\cd_2$ is always bounded by $1$.
We note that our notion of $2$-dimension is similar, but not quite identical to the one of \cite[220]{Kato_GaloisCohomologyCompleteValuedFields};
for instance, to an algebraically closed field of characteristic $2$
we assign $2$-dimension $1$ whereas Kato would assign $2$-dimension $0$.
Our notion is designed ad hoc to simplify our statements.

Let us note that $2$-dimension behaves well with regard to finitely generated field extensions,
in the sense that for finitely generated $L/K$ with $\dim_2 K < \infty$
we have $\dim_2 L = \dim_2 K + r$ where $r$ is the transcendence degree of $L$ over $K$.
In case $\charac(K) \neq 2$, this is \cite[Theorem 6.5.14]{Neu15},
and otherwise \cite[Chapitre V, §16, No 6, Corollaire 3 to Théorème 4]{Bourbaki_AlgebreII}.

The main theorem of this section is as follows.
It concerns \emph{function fields in one variable} $F/K$, i.e.\ finitely generated field extensions of transcendence degree $1$.
Given such a function field, we write $\Val(F/K)$ for the set of $\zz$-valuations of $F$ trivial on $K$.
Such $\zz$-valuations are sometimes also called \emph{places}.
We are mostly (but not exclusively) interested in those function fields in one variable which are regular.
Recall here that a field extension $F/K$ is called \emph{regular} if $K$ is relatively algebraically closed in $F$ and $F/K$ is a separable extension;
see \cite[Section 2.6]{Fri08} for properties and equivalent characterisations of regular field extensions.
\begin{stel}\label{T:Reciprocity-Main}
  Let $K$ be a field whose $2$-dimension $d = \dim_2 K$ satisfies $1 \leq d < \infty$.
  Let $F/K$ be a regular function field in one variable.
  Then the maps
  \[ I_q^{d+1}F \to \bigoplus_{v \in \Val(F/K)} I_q^d(Fv) \to I_q^d K\]
  form an exact sequence,
  where the first map is the sum of residue maps
  and the second map is the sum of transfer maps induced by any choice of non-zero $K$-linear maps $Fv \to K$.
\end{stel}

Let us comment on the statement of the theorem.
Since $K$ has $2$-dimension $d$ and thus $F$ has $2$-dimension $d+1$, we have $I_q^{d+1}K = 0$ and $I_q^{d+2}F = 0$
(and likewise $I_q^{d+1}(Fv) = 0$ for the finite extensions $Fv$ of $K$):
this follows from \Cref{P:cdCharacterisation} if $\charac(K) \neq 2$,
and if $\charac(K) = 2$ it follows from the fact that all $d$-fold bilinear Pfister forms over $K$ are isotropic due to the dimension condition \cite[Example 6.5]{ElmanKarpenkoMerkurjev}.
Therefore the residue maps defined in \Cref{C:residueHomomorphism} make sense as maps $\partial_v \colon I_q^{d+1}F \to I_q^d(Fv)$,
in which form they occur in \Cref{T:Reciprocity-Main}.

Furthermore, for a finite field extension $L/K$ with $\dim_2 K = d$,
\Cref{lem:transfer-unique} implies that for any choice of non-zero $K$-linear map $L \to K$,
the induced transfer $I_q L \to I_q K$ restricts to a homomorphism $I_q^d L \to I_q K$
independent of the choice of map $L \to K$.
We will occasionally simply refer to this as \emph{the} transfer map (on $I_q^d$).
In fact, we will see in \Cref{lem:transfer-surj-kernel} that this transfer on $I_q^d L$
has image $I_q^d K$, so we may consider it as a surjective map $I_q^d L \to I_q^d K$.
The maps $I_q^d(Fv) \to I_q^d K$ in \Cref{T:Reciprocity-Main} arise in this way.

  The fact that the two homomorphisms from \Cref{T:Reciprocity-Main}
  form a complex is relatively easy to show for arbitrary $K$ of characteristic not $2$,
  and in this case is true for arbitrary $d$.
  This is related to the more general statement in Milnor $K$-theory known as
  Weil Reciprocity, see for instance \cite[Proposition 7.4.4]{GilleSzamuely_2nd}.
  The main work in the proof of the theorem is therefore to show that the two maps not only form a complex,
  but an exact sequence,
  and to handle the case of characteristic $2$.

\begin{opm}\label{rem:reciprocity-discussion}
  If $K$ is a finite field, and therefore $\dim_2 K =1$,
  it follows from the classical theory of quadratic forms over global fields
  that there is an exact sequence
  \[ 0 \to I_q^2 F \to \bigoplus_{v \in \Val(F/K)} I_q^1(Fv) \to I_q^1 K \to 0 \]
  with maps as in \Cref{T:Reciprocity-Main}.
  (In fact, the terms of this sequence are compatibly isomorphic to
  the $2$-torsion part of the Hasse--Brauer--Noether sequence
  from global class field theory \cite[Theorem 8.1.17]{Neu15}.)
  Therefore the theorem holds in this case,
  and additionally the homomorphism $I_q^2 F \to \bigoplus_{v \in \Val(F/K)} I_q^1(Fv)$ is injective.

  Let us stress that unlike in the case of a finite base field $K$,
  injectivity of the first homomorphism does not usually hold in \Cref{T:Reciprocity-Main}.
  In the case of $2$-dimension $1$, this is related to the non-vanishing of the
  so-called unramified Brauer group of the function field $F$.
  To obtain an example,
  let $K$ be a $C_1$-field such that $\charac(K) \neq 2$,
  $I^1_q K \neq 0$, and $K$ carries a $\zz$-valuation (e.g.~$K = \cc(T)$, or $K$ is the maximal unramified extension of $\qq_p$). Then $\cd_2(K) = 1$.
  As explained in \cite[Example 6.11]{SQRP} there exists an anisotropic quadratic form in $4$ variables over $F = K(T)$ which is isotropic over $F_v$ for all $v \in \Val(F/K)$.
  By \cite[Proposition 6.1.3]{GuptaThesis} there exists then also a function field in one variable $F'/K$ and an anisotropic $2$-fold Pfister form $q$ over $F'$ such that $q_{F'_v}$ is isotropic for all $v \in \Val(F'/K)$.
  This implies that $0 \neq [q] \in I^2_q F'$ but $\partial_v[q] = 0$ for all $v \in \Val(F'/K)$.
\end{opm}
\begin{opm}
In \cite[Proposition 7.2.15]{DaansThesis} a proof of \Cref{T:Reciprocity-Main} is given in the case where the characteristic is not $2$ in the language of Galois cohomology.
The proof given there follows several of the main ideas also to be found in the present exposition, but is notably shorter as it can rely on the more well-developed relevant theory and results in the literature in this case.
\end{opm}

We begin by establishing some basic properties of the transfer map,
culminating in \Cref{lem:transfer-surj-kernel} and \Cref{lem:residue-transfer-commute} below.
For proving these properties, we occasionally use the language of Galois cohomology,
since the groups $I_q^d$ may be identified with certain Galois cohomology groups.
We assume basic facts about Galois cohomology groups, as well as restriction an corestriction maps,
as found in \autocite[Chapter I]{Neu15} or \autocite[§99]{ElmanKarpenkoMerkurjev}.
\begin{lem}\label{lem:transfer-insep-iso}
  Let $L/K$ be a finite purely inseparable field extension
  where $d= \dim_2 K$ satisfies $1 \leq d < \infty$.
  The transfer $I_q^d L \to I_q K$ is injective with image $I_q^d K$.
  If $\charac(K) \neq 2$, the resulting isomorphism $I_q^d L \to I_q^d K$ is inverse to restriction.
\end{lem}
\begin{proof}
  Let us first suppose that $\charac(K) \neq 2$.
  Then the restriction $I_q K \to I_q L$ is an isomorphism,
  and for a suitable choice of functional $L \to K$, the induced transfer map
  $I_q L \to I_q K$ is inverse to the restriction
  (see for example \autocite[beginning of Section 2 \& middle of p.~458]{Arason_CohomInvarianten}).
  Since restriction and transfer are compatible with the filtration of $I_q$ \cite[Satz 3.3]{Arason_CohomInvarianten},
  this induced transfer map in particular gives an isomorphism $I_q^d L \to I_q^d K$.

  Let us now suppose that $\charac(K) = 2$.
  Since $L/K$ decomposes into inseparable quadratic steps,
  we may suppose that $L/K$ is inseparable quadratic.
  The image of $I_q^d L$ under the transfer map is contained in $I_q^d K$ by \autocite[Corollary~34.17~(2)]{ElmanKarpenkoMerkurjev}.
  We now transform the problem of proving that the transfer map $I_q^d L \to I_q^d K$ is an isomorphism
  into a problem about a bilinear transfer map, as follows.
  We have the following diagram:
    \[ \xymatrix{
       I_q^d L \ar[r] \ar[d] & I_q^d K \ar[d] \\
       H^1(L, K^M_{d-1}(L^{\mathrm{sep}})/2K^M_{d-1}(L^{\mathrm{sep}}))  \ar[r] \ar[d] & H^1(K, K^M_{d-1}(K^{\mathrm{sep}})/2K^M_{d-1}(K^{\mathrm{sep}})) \ar[d] \\
       H^1(L, I^{d-1} L^{\mathrm{sep}}) \ar[r] & H^1(K, I^{d-1} K^{\mathrm{sep}})
     }\]
  Here $K^M_n$ stands for the $n$-th Milnor K-group, and $K^{\mathrm{sep}} \subseteq L^{\mathrm{sep}}$ are separable closures of $K$ and $L$, respectively.
  We will henceforth, in comparing with the statements of \autocite{ElmanKarpenkoMerkurjev},
  frequently use that $I_q^{d+1} K = 0$ and $I^d K^{\mathrm{sep}} = 0$ (and similarly for $L$, $L^{\mathrm{sep}}$)
  since $d = \dim_2 K$ and therefore all $d$-fold bilinear Pfister forms are isotropic over all fields under consideration
  \autocite[Example~6.5]{ElmanKarpenkoMerkurjev}.

  The upper square is the commutative square of \autocite[Lemma~40.2]{ElmanKarpenkoMerkurjev}.
  Here the middle horizontal arrow is induced by the restriction isomorphism of absolute Galois groups $\operatorname{Gal}(L^{\mathrm{sep}}/L) \to \operatorname{Gal}(K^{\mathrm{sep}}/K)$
  and the norm map $K^M_{d-1}(L^{\mathrm{sep}}) \to K^M_{d-1}(K^{\mathrm{sep}})$ in K-theory,
  and the vertical arrows are isomorphisms.%
  \footnote{In \autocite[Lemma~40.2]{ElmanKarpenkoMerkurjev}, the terms in the middle row of our diagram are labelled $H^n(L)$ and $H^n(K)$
    (where $n = d$),
    or $H^{n,n-1}(L)$ and $H^{n,n-1}(K)$ later in the book.
    Unfortunately, there is an inconsistency in the use of this notation,
    given that $H^n(L)$ stands for $H^{n,n}(L)$ in \autocite[§101]{ElmanKarpenkoMerkurjev}.
    This was already observed in \autocite[footnote 2 on p.~1207]{DittmannPop_v1}.
  }
  The lower square is induced in cohomology by the commutative square of \autocite[Lemma~40.1]{ElmanKarpenkoMerkurjev}.
  The vertical arrows in our lower square are thus isomorphisms induced by
  isomorphisms $K^M_{d-1}(L^{\mathrm{sep}})/2K^M_{d-1}(L^{\mathrm{sep}}) \to I^{d-1} L^{\mathrm{sep}}$
  and similarly for $K^{\mathrm{sep}}$, see \autocite[Fact~5.15]{ElmanKarpenkoMerkurjev}.
  The bottom horizontal arrow is induced by the restriction isomorphism of absolute Galois groups and
  the bilinear transfer map $I^{d-1} L^{\mathrm{sep}} \to I^{d-1} K^{\mathrm{sep}}$ with respect to an arbitrary
  non-zero functional $L^{\mathrm{sep}} \to K^{\mathrm{sep}}$.

  Since the entire diagram commutes and all vertical arrows are isomorphisms,%
  \footnote{The vertical isomorphisms seem to be the same as the horizontal arrows
    in the following \Cref{lem:kato-isos-Ir} obtained from \autocite{Kato_Char2}.
    However, this is not so easy to see from the sources cited, and so we do not use it.}
  it suffices to prove that the bottom horizontal arrow is an isomorphism,
  or indeed that the bilinear transfer map $I^{d-1} L^{\mathrm{sep}} \to I^{d-1} K^{\mathrm{sep}}$ is an isomorphism.

  This holds by \autocite[Theorem 40.6]{ElmanKarpenkoMerkurjev},
  since the restriction $I^{d-1}K^{\mathrm{sep}} \to I^{d-1}L^{\mathrm{sep}}$ is the zero map:
  Indeed, by the assumption that $[K : K\pow{2}] = 2^{d-1}$, no $d-1$ elements of $K^{\mathrm{sep}}$
  which are $2$-independent can remain so in $L$ as $L/K$ is inseparable,
  and so every anisotropic $(d-1)$-fold bilinear Pfister form in $K^{\mathrm{sep}}$
  must become isotropic in $L^{\mathrm{sep}}$ by \autocite[Example 6.5]{ElmanKarpenkoMerkurjev}.
\end{proof}

\begin{lem}\label{lem:kato-isos-Ir}
  Let $K$ be a field of characteristic $2$.
  For every $d > 0$, there is a natural isomorphism
  $I_q^d K \cong H^1(K, I^{d-1}K^{\mathrm{sep}})$.
  These isomorphisms are compatible for different $d$, in the sense that for $d \geq d'$
  the diagram
  \[ \xymatrix{
      I_q^d K \ar[r] \ar[d] & H^1(K, I^{d-1}K^{\mathrm{sep}}) \ar[d] \\
      I_q^{d'} K \ar[r] & H^1(K, I^{d'-1}K^{\mathrm{sep}})
    }\]
  commutes, where the left-hand vertical map is the natural inclusion,
  and the right-hand vertical map is the map induced in cohomology
  by the inclusion $I^{d-1}K^{\mathrm{sep}} \hookrightarrow I^{d'-1}K^{\mathrm{sep}}$.
\end{lem}
\begin{proof}
  The existence of the isomorphisms is stated in \autocite[p.~506, Corollary]{Kato_Char2}.
  Functoriality and compatibility for different $d$ are not stated explicitly there,
  but follow from the proof.
\end{proof}

\begin{lem}\label{lem:transfer-char2-filtered-cor}
  Let $L/K$ be a finite separable extension of fields of characteristic $2$,
  let $\Omega$ be a separable closure of $L$ (hence also of $K$),
  and let $d > 0$.

  The transfer $I_q L \to I_q K$ induced by the trace map $L \to K$ restricts to a map $I_q^d L \to I_q^d K$.
  Furthermore under the identification of \Cref{lem:kato-isos-Ir},
  the latter map corresponds to the corestriction $\operatorname{cor} \colon H^1(L, I^{d-1}\Omega) \to H^1(K, I^{d-1}\Omega)$.
\end{lem}
\begin{proof}
  Consider the commutative square
  \[\xymatrix{
      H^1(L, I^{d-1}\Omega) \ar[d]^{\operatorname{cor}} \ar[r] &H^1(L, W\Omega) \ar[d]^{\operatorname{cor}} \\
      H^1(K, I^{d-1}\Omega) \ar[r] & H^1(K, W\Omega).
    }\]
  By the compatibility statement of \Cref{lem:kato-isos-Ir},
  the terms of the square are compatibly isomorphic with the terms of the following square:
  \[ \xymatrix{
      I_q^d L \ar@{^{(}->}[r]\ar[d] & I_q L \ar[d] \\
      I_q^d K \ar@{^{(}->}[r] & I_q K
    }\]
  The right-hand vertical arrow is precisely the transfer induced by the trace $L \to K$
  \autocite[Satz 9]{Arason_WittCohomChar2}.
  It follows that the left-hand vertical arrow
  must be the restriction of this transfer $I_qL \to I_qK$.
\end{proof}

\begin{lem}\label{lem:transfer-surj-kernel}
  Let $L/K$ be a finite field extension
  where $d = \dim_2 K$ satisfies $1 \leq d < \infty$.

  The image of the transfer map $I_q^d L \to I_q K$ is precisely $I_q^d K$.
  If $L/K$ is normal, then the kernel of this map is generated as an abelian group
  by expressions $\sigma q - q$, where $q \in I_q^d L$ and $\sigma \in \operatorname{Aut}(L/K)$.
\end{lem}
\begin{proof}
  Since every finite field extension embeds into a finite normal extension,
  and the transfer is functorial and unique by \Cref{lem:transfer-unique}, it suffices to prove everything in the
  normal case.
  Since a normal extension $L/K$ decomposes as a Galois extension $L/L_0$ and
  a purely inseparable extension $L_0/K$, which we analysed in \Cref{lem:transfer-insep-iso},
  we may assume that $L/K$ is a finite Galois extension.

  Let us first suppose that $\charac K \neq 2$.
  By the Milnor Conjectures (Theorem of Orlov--Vishik--Voevodsky, see e.g.~\cite[Fact 16.2]{ElmanKarpenkoMerkurjev}),
  there is an isomorphism of $I_q^d K$ with $H^d(K, \Z/2)$, and similarly for $L$.
  By \autocite[Satz 4.18]{Arason_CohomInvarianten}, under these isomorphisms
  the transfer corresponds to the corestriction map
  $\operatorname{cor} \colon H^d(L, \Z/2) \to H^d(K, \Z/2)$.
  Since $d = \cd_2(K)$ by assumption, this corestriction is surjective
  and its kernel is generated by expressions of the required form
  by \autocite[Proposition 3.3.11]{Neu15}.
  
  We now consider the case of characteristic $2$.
  By \Cref{lem:transfer-char2-filtered-cor},
  instead of the transfer $I_q^d L \to I_q^d K$ we may consider the corestriction
  $\operatorname{cor} \colon H^1(L, I^{d-1}\Omega) \to H^1(K, I^{d-1}\Omega)$,
  where $\Omega$ is a separable closure of $L$.
  This corestriction is surjective,
  and its kernel is generated by elements of the required form \autocite[Proposition 3.3.11]{Neu15}.
\end{proof}

\begin{opm}\label{opm:aravire-baeza}
  For a finite separable extension $L/K$ in characteristic $2$,
  the fact that the image of the transfer map $I_q^d L \to I_q K$ is
  precisely $I_q^d K$ was also shown in \cite[Theorem~2.7]{Aravire-Baeza_v-invariant}
  without cohomological methods.
  There it is even shown that this holds for any integer $d$ (without assumption on $\dim_2 K$),
  for the transfer map $s^\ast$ associated to
  an arbitrary non-zero $K$-linear $s \colon L \to K$.
\end{opm}

We need the following result on compatibility of transfer and residue maps.
Recall here that for $\zz$-valued fields $(K,v)$ and $(L, w)$ with $K \subseteq L$,
we say that $w$ \emph{lies over} $v$ if $\mathcal{O}_w \cap K = \mathcal{O}_v$.
In this situation, the \emph{ramification index} $e(w/v)$ is the positive natural number
such that $w(x) = e(w/v) \cdot v(x)$ for all $x \in K$.
\begin{lem}\label{lem:residue-transfer-commute}
Let $(K,v)$ be a $\zz$-valued field, $L/K$ a finite extension,
and $w_1, \dotsc, w_r$ the finitely many $\zz$-valuations of $L$ lying over $v$.
Suppose that $[L : K] = \sum_{i=1}^r [Lw_i : Kv] \cdot e(w_i/v)$
(this is sometimes called the \emph{fundamental equality}).
Assume that $1 \leq d = \dim_2(Kv) < \infty$.
Then the diagram
  \[ \xymatrix{
      I_q^{d+1} L \ar[r] \ar[d] & \bigoplus_{i=1}^r I_q^d (Lw_i) \ar[d] \\
      I_q^{d+1} K \ar[r] & I_q^d (Kv)
    }\]
commutes, where the vertical arrows are (sums of) transfer maps,
the bottom horizontal arrow is the residue homomorphism with respect to $v$,
and the top horizontal arrow is the sum of the residue homomorphisms with respect to the $w_i$.
\end{lem}

\begin{opm}\label{opm:residue-transfer-commute-nondyad}
  Although we give a full proof below in all cases,
  we observe straight away that \Cref{lem:residue-transfer-commute}
  for a non-dyadic valuation $v$ is given by \cite[Satz 3.4]{Arason_CohomInvarianten}
  (and our proof below shares some similarities with the proof there).
  This non-dyadic case suffices to establish \Cref{T:Reciprocity-Main} in characteristic not $2$.
  Furthermore, the hypothesis that the fundamental equality holds may be omitted in this case.

  See also \cite[Corollary 7.4.3]{GilleSzamuely_2nd} for a statement in Milnor $K$-theory
  which implies \Cref{lem:residue-transfer-commute} for non-dyadic valuations
  (using the correspondence of residue maps of Milnor $K$-groups and quadratic forms
  discussed in \Cref{rem:residue-K-cohom}).
\end{opm}

\begin{proof}[Proof of \Cref{lem:residue-transfer-commute}]
Let us first assume that $K$ is complete with respect to $v$.
In this case $v$ is henselian, and there is precisely one $\zz$-valuation $w = w_1$ of $L$ lying over $v$,
the fundamental equality $[L : K] = [Lw : Kv] \cdot e(w/v)$ always holds,
and $\mc{O}_w$ is a free $\mc{O}_v$-module
\cite[Chapitre 6, §8, No 5, Corollaire 2 to Théorème 2]{Bou06}.
We consider three special cases.

\textit{Case 1: Assume that $e(w/v) = 1$}, i.e.~a uniformiser for $K$ is also a uniformiser for $L$.
Since $\mc{O}_w$ is a free $\mc{O}_v$-module,
we can find an $\mc{O}_v$-linear map $s : \mc{O}_w \to \mc{O}_v$ such that $\ovl{s} : Lw \to Kv : x \mapsto \ovl{x}$ is non-zero.
We can extend $s$ to a $K$-linear map $L \to K$, which we also denote by $s$.
By \Cref{L:transfer-unimodular} we have for every unimodular anisotropic quadratic form $q$ over $L$ that transfer and residue maps commute.
Now consider an arbitrary anisotropic $(d+1)$-fold Pfister form $q$ over $L$.
In view of \Cref{C:Pfisterresidue} we have $[q] = [q'] - [\pi q'] = [\llangle \pi \rrangle^b_L \otimes q']$ for some unimodular form $q'$ and a uniformiser $\pi$, which by the assumption $e(w/v) = 1$ we may choose to lie in $K$. 
Denoting by $\tilde{q}$ the anisotropic form in the Witt equivalence class $s^\ast [q']$ (which is unimodular in view of \Cref{L:transfer-unimodular}), we compute that
\begin{align*}
\partial_v s^\ast[q] &= \partial_v s^\ast [\llangle \pi \rrangle^b_L \otimes q'] = \partial_v([\llangle \pi \rrangle^b_K] \otimes s^\ast [q'])  \\
&= \partial_v([\llangle \pi \rrangle^b_K] \otimes [\tilde{q}]) = [\RF{1}{\tilde{q}}] = \ovl{s}^\ast [\RF{1}{q'}] = \ovl{s}^\ast \partial_w [q],
\end{align*}
where in the second equality we applied Frobenius Reciprocity \cite[Proposition 20.2(20.3b)]{ElmanKarpenkoMerkurjev}, in the fourth we applied \Cref{C:Pfisterresidue}, and in the fifth we invoked \Cref{L:transfer-unimodular}.

\textit{Case 2: Assume that $Lw/Kv$ is purely inseparable of odd degree}.
Let $q$ be an anisotropic $(d+1)$-fold Pfister form over $L$; as before, we may assume $q = \llangle a_1, \ldots, a_d, b]]_L$ with $w(a_1) = \ldots = w(a_{d-1}) = w(b) = w(1+4b) = 0$, and that $a_d$ is a uniformiser for $w$.
By \Cref{C:residueHomomorphism}, the isometry class of $q$ is completely determined by that of its residue form $\RF{1}{q} \cong \llangle \ovl{a_1}, \ldots, \ovl{a_{d-1}}, \ovl{b}]]_{Lw}$ (recall that we still assume that $v$ is complete, in particular henselian, whereby the same holds for $w$).
The assumption that $Lw/Kv$ is purely inseparable of odd degree implies that this residue form is actually defined over $Kv$: indeed, either $Kv = Lw$ and there is nothing to show, or
$p = \charac(Kv)$ is odd and $[Lw: Kv] = p^k$ for some $k \in \nat$,
and so $\llangle \ovl{a_1}, \ldots, \ovl{a_{d-1}}, \ovl{b}]]_{Lw} \cong \llangle \ovl{a_1}^{p^k}, \ldots, \ovl{a_{d-1}}^{p^k}, \ovl{b}^{p^k}]]_{Lw}$.
Since $\beta^{p^k} \in Kv$ for every $\beta \in Lw$, we may thus as well assume $a_1, \ldots, a_{d-1}, b \in \mc{O}_v$.

We now use the fact that $s^\ast([\llangle a \rrangle^b_F]) = [\llangle N_{F/E}(a) \rrangle^b_E]$ for every field extension $F/E$, $a \in F^\times$, and well-chosen functional $s$ \cite[Corollary 20.13]{ElmanKarpenkoMerkurjev};
furthermore, observe that $v(N_{L/K}(a_d)) = [Lw : Kv]$ \cite[Chapitre II, §2, Corollaire 4]{CorpsLocaux} is odd.
We now use again Frobenius Reciprocity \cite[Proposition 20.2(20.3c)]{ElmanKarpenkoMerkurjev} and \Cref{P:Pfisterresidue} to compute:
\begin{multline*}
  \partial_v(s^\ast([q]))
  = \partial_v(s^\ast([\llangle a_d \rrangle^b_L \otimes \llangle a_1, \ldots, a_{d-1}, b]]_L]) \\
  = \partial_v([\llangle N_{L/K}(a_d) \rrangle^b_K \otimes \llangle a_1, \dotsc, a_{d-1}, b]]_K])
  = [\llangle \ovl{a_1}, \dotsc, \ovl{a_{d-1}}, \ovl{b} ]]_{Kv} ]
\end{multline*}
On the other hand, we also have
\[ \ovl{s}^\ast(\partial_w [q])
  = \ovl{s}^*([ \llangle \ovl{a_1}, \dotsc, \ovl{a_{d-1}}, \ovl{b}]]_{Lw}])
  = [\llangle \ovl{a_1}, \dotsc, \ovl{a_{d-1}}, \ovl{b} ]]_{Kv} ],\]
where the last equality is trivial if $Lw = Kv$ (for instance if $\charac(Kv) = 2$),
and otherwise follows from the fact that transfer and restriction are inverse maps
(\Cref{lem:transfer-insep-iso}).

\textit{Case 3: Assume that $Lw/Kv$ is purely inseparable of even degree}.
In this case $\charac(Kv) = 2$.
Let $M/K$ be a normal closure of $L/K$ and extend $w$ to $M$.
By a standard Galois theory argument using the fixed field of a $2$-Sylow subgroup of $\operatorname{Aut}(M/K)$,
we may decompose $M/K$ as a tower of field extensions $K = K_0 \subseteq K_1 \subseteq \ldots \subseteq K_m = M$ where each $K_{i+1}/K_i$ is either quadratic or of odd degree.
Each odd degree extension among the $K_{i+1}/K_i$ can be split into an unramified and a totally ramified part
\cite[Chapitre III, §5, Corollaire 3 to Théorème 3]{CorpsLocaux},
so we may assume that every extension $K_{i+1}/K_i$ falls under either case 1 or case 2
(this is automatic for the quadratic extensions).
In particular, we know the the induced diagram commutes for each of these.
By functoriality of the transfer maps, the same then holds for the extension $M/K$, and by the surjectivity of the transfer map $I^{d+1}_qM \to I^{d+1}_qL$ (\Cref{lem:transfer-surj-kernel}) we conclude that also for $L/K$ the diagram commutes.

\textit{General case for $(K,v)$ complete:}
The extension of complete $\zz$-valued fields $L/K$ has a maximal unramified subextension.
This is an intermediate field $K'$, with $\zz$-valuation $v'$ lying under $w$,
such that $K'/K$ is unramified (i.e.~$e(v'/v) = 1$ and $K'v'/Kv$ is separable)
and $Lw/K'v'$ is purely inseparable
\cite[Chapitre III, §5, Corollaire 3 to Théorème 3]{CorpsLocaux}.
The result thus follows from the three special cases mentioned above and functoriality of the transfer maps.

\emph{We now drop the condition that $(K,v)$ is complete.}
We denote by $\widehat{K}$ the completion of $K$ with respect to $v$,
and by $\widehat{L}_i$ the completion of $L$ with respect to $w_i$, where $i = 1, \dotsc, r$.
We continue to write $v$ and $w_i$ for the natural prolongations of these valuations to the corresponding completions.
By the first part of the proof, we have a commutative diagram
  \[ \xymatrix{
      \bigoplus_{i = 1}^r I_q^{d+1} \widehat{L}_i \ar[r] \ar[d] & \bigoplus_{i=1}^r I_q^d (Lw_i) \ar[d] \\
      I_q^{d+1} \widehat{K} \ar[r] & I_q^d (Kv)
    }\]
where the vertical arrows are (sums of) transfer maps
and the horizontal arrows are (sums of) residue maps.
We furthermore have a commutative diagram
\[ \xymatrix{
    I_q^{d+1} K \ar[r] \ar[d] & I_q^d (Kv) \\
    I_q^{d+1} \widehat{K} \ar[ur]
  }\]
where the horizontal and diagonal arrows are residue maps with respect to $v$
and the vertical arrow is restriction.
Here commutativity can be immediately read off from the concrete description of residue maps in \Cref{P:Pfisterresidue}
(which characterises residue maps by \Cref{P:Pfister-semilocal-presentation}).
We have the analogous diagram for each $(L,w_i)$ in place of $(K,v)$.

Let us now argue that the diagram
\[ \xymatrix{
    I_q^{d+1} L \ar[r] \ar[d] & \bigoplus_{i=1}^r I_q^{d+1} \widehat{L}_i \ar[d] \\
    I_q^{d+1} K \ar[r] & I_q^{d+1} \widehat{K}
  }\]
commutes,
where the vertical maps are (sums of) transfers and the horizontal maps are (sums of) restrictions.
Let $s \colon L \to K$ be a non-zero $K$-linear map.
We obtain an induced $\widehat{K}$-linear map $\widehat{s} \colon \widehat{K} \otimes_K L \to \widehat{K}$.
By \cite[Chapitre 6, §8, No 5, Corollaire 3 to Théorème 2]{Bou06}
(using the fundamental equality),
we have a canonical $K$-algebra isomorphism $\widehat{K} \otimes_K L \to \prod_{i=1}^r \widehat{L}_i$,
induced by the natural $K$-linear embeddings of $\widehat{K}$ and $L$ into each $\widehat{L}_i$.
Therefore $\widehat{s}$ yields a collection of $\widehat{K}$-linear maps $s_i \colon \widehat{L}_i \to \widehat{K}$, $i = 1, \dotsc, r$.
For sufficiently generic choice of $s$, none of the $s_i$ will be zero:
This is a simple matter of ensuring that the kernel of $\widehat{s}$ does not contain the preimage of any of the $\widehat{L}_i$.
It is now routine to check
(cf.\ \cite[Satz 2.2]{Arason_CohomInvarianten} in characteristic not $2$)
that starting with a quadratic form $q$ over $L$,
we obtain the same quadratic form over $\widehat{K}$ by first transferring to $K$ using $s$ and then restricting to $\widehat{K}$,
or by first restricting to $\widehat{L}_i$, transferring to $\widehat{K}$ using $s_i$, and summing over all $i$.

We have now proven that in the diagram
\[ \xymatrix{
    I_q^{d+1} L \ar[rr] \ar[rd] \ar[ddd] & & \bigoplus_{i=1}^r I_q^d (Lw_i) \ar[ddd] \\
    & \bigoplus_{i=1}^r I_q^{d+1} \widehat{L}_i \ar[ur] \ar[d] & \\
    & I_q^{d+1} \widehat{K} \ar[dr] & \\
    I_q^{d+1} K \ar[rr] \ar[ur] & & I_q (Kv)
  }\]
(where each arrow is a transfer map, restriction map, residue map, or sum thereof)
the left and right trapezium commute, and the top and bottom triangle commute.
Therefore the outside square commutes, which was to be shown.
\end{proof}

This finishes our preparations concerning the transfer map.
The following lemma on rational function fields will be used to show that the maps from \Cref{T:Reciprocity-Main}
form a complex in characteristic $2$.
The analogous result in characteristic away from $2$ is standard
\autocite[Corollary 21.7]{ElmanKarpenkoMerkurjev} and needs no hypothesis on $d$.
\begin{lem}\label{lem:Milnor-seq-char-2}
Let $K$ be a field of characteristic $2$, $F = K(T)$,
and suppose that $d = \dim_2 K$ satisfies $1 \leq d < \infty$.
Then the sequence
$$ I^{d+1}_q F \xrightarrow{\oplus \partial_v} \bigoplus_{v \in \mathcal{V}(F/K)} I^{d}_q (Fv) \xrightarrow{\oplus s_v^\ast} I^d_q K $$
is a complex,
where the first arrow is the sum of residue maps as defined in \Cref{C:residueHomomorphism},
and the second arrow is the sum of transfer maps.
\end{lem}
\begin{proof}
We deduce this from the results of the paper \cite{Aravire-Jacob}.
By \cite[Theorem~6.2]{Aravire-Jacob}, we have an exact sequence
\begin{equation}\label{Aravire-Jacob-sequence}
I_q F \xrightarrow{\oplus \partial_v^{AJ}} \bigoplus_{v \in \mc{V}(F/K)} W_1 \widehat{F_v} \xrightarrow{\oplus s_v^{AJ}} I_q K.
\end{equation}
We will explain the terms of the sequence below.
For our purposes, it suffices to show that we have a commutative diagram of group homomorphisms
\begin{equation}\label{eq:diagram-two-MilnorES}
\begin{tikzcd}
	{I^{d+1}_q F} & {\bigoplus_{v \in \Val(F/K)}I^d_q (Fv)} & {I^d_qK} \\
	{I_q F} & {\bigoplus_{v \in \Val(F/K)} W_1 \widehat{F_v}} & {I_q K}
	\arrow["{\oplus \partial_v^{AJ}}", from=2-1, to=2-2]
	\arrow["{\oplus s_v^{AJ}}", from=2-2, to=2-3]
	\arrow[hook, from=1-1, to=2-1]
	\arrow[hook, from=1-3, to=2-3]
	\arrow["{\oplus \partial_v}", from=1-1, to=1-2]
	\arrow["\iota"', from=1-2, to=2-2]
	\arrow["{\oplus s_v^\ast}", from=1-2, to=1-3]
\end{tikzcd}
\end{equation}
where the unnamed vertical arrows are inclusions and $\iota$ is a group homomorphism which we will construct.
Once we have this, the fact that the upper row is a complex follows from the fact that the lower row is a complex (even exact, by \eqref{Aravire-Jacob-sequence}).

Let us first focus on the left-hand square of \eqref{eq:diagram-two-MilnorES}.
For $v \in \Val(F/K)$, we write $\widehat{F_v}$ for the completion of $F$ with respect to $v$.
There exists a (non-unique) lifting $Fv \to \widehat{F_v}$,
i.e.\ an embedding of fields which is right inverse to the residue homomorphism.
This induces a homomorphism $I_q(Fv) \to I_q \widehat{F_v}$.
This homomorphism is an embedding and does not in fact depend on the choice of lifting $Fv \to \widehat{F_v}$,
since the image of the class of an anisotropic form $q$ in $I_q(Fv)$
is the class of an anisotropic unimodular form over $\widehat{F_v}$ with first residue form $q$,
and such a unimodular form is unique by \autocite[Satz~4.1]{Tietze}.
(This canonicity of $I_q(Fv) \to I_q \widehat{F_v}$ seems to have not been observed in \autocite{Aravire-Jacob},
but is given in the later \autocite[Corollary~1.3]{Aravire-Jacob_hyperelliptic} by the same authors,
using an explicit calculation.)
The group $W_1 \widehat{F_v}$ is now simply defined as the cokernel of $I_q(Fv) \to I_q \widehat{F_v}$,
and the residue map $\partial_v^{AJ} \colon I_q F \to W_1 \widehat{F_v}$ is the composition of the homomorphism
$I_q F \to I_q \widehat{F_v}$ induced by $F \hookrightarrow \widehat{F_v}$ and the canonical map $I_q \widehat{F_v} \to W_1 \widehat{F_v}$
\autocite[20]{Aravire-Jacob}.

We construct the homomorphism $\iota$ of \eqref{eq:diagram-two-MilnorES} summand-by-summand.
For $v \in \Val(F/K)$ with a fixed uniformiser $\pi$
we map $I_q^d(Fv)$ to $W_1 \widehat{F_v}$ by sending the class of a form $q$
to the class of $\pi q$ (where we see $q$ as a form on $\widehat{F_v}$ by means of a lifting $Fv \to \widehat{F_v}$,
and the choice of lifting does not matter by the above).
We note in passing that this homomorphism $I_q^d(Fv) \to W_1 \widehat{F_v}$ (and therefore also $\iota$)
is in fact injective,
as is easy to see for instance by the decomposition theorem \autocite[Theorem~1.3]{Aravire-Jacob} for $I_q \widehat{F_v}$,
but we will not use this.
Let us furthermore observe that our homomorphism $I_q^d(Fv) \to W_1 \widehat{F_v}$ does not depend on the choice of the uniformiser $\pi$.
Indeed, if $\pi'$ is another uniformiser, then we have $\partial_v([q] - [\pi q]) = \partial_v([q] - [\pi' q])$ by \Cref{C:Pfisterresidue}
and therefore $[q] - [\pi q] = [q] - [\pi' q]$ in $I_q^{d+1} \widehat{F_v}$ by \Cref{C:residueHomomorphism},
so $[\pi q] = [\pi' q]$ in $I_q \widehat{F_v}$ and a fortiori in $W_1 \widehat{F_v}$.

Let us now show that the left-hand square of \eqref{eq:diagram-two-MilnorES} commutes.
Since $I^{d+1}_q F$ is generated by the classes of anisotropic $(d+1)$-fold Pfister forms,
it suffices to consider one such form $q$.
Let $v \in \Val(F/K)$.
If $\partial_v[q] = 0$, then by \Cref{C:residueHomomorphism} the class of $q$ in $I_q \widehat{F_v}$ is zero, and since $\partial_v^{AJ}$ is defined as the composition of the natural homomorphisms $I_q F \to I_q \widehat{F_v}$ and $I_q \widehat{F_v} \to W_1 \widehat{F_v}$, also $\partial_v^{AJ}[q] = 0$.
We may thus assume from now on that $\partial_v[q] \neq 0$.
By \Cref{P:Pfister-semilocal-presentation}, we have a representation $q \cong \llangle a_1, \dotsc, a_d, b]]_F$
with $a_1, \dotsc, a_{d-1}, b \in \mathcal{O}_v^\times$ and $v(a_d) \in \{ 0, 1\}$.
By \Cref{P:Pfisterresidue} and since we assumed $\partial_v[q] \neq 0$, in fact $a_d$ is a uniformiser for $v$.
Now $[q] = [q'] - [a_d q']$ for the form $q' = \llangle a_1, \dotsc, a_{d-1}, b]]_F$.
The image of $[q]$ under $\partial_v \colon I_q^{d+1} F \to I_q^d(Fv)$ is $\llangle \overline{a_1}, \dotsc, \overline{a_{d-1}}, \overline{b}]]_{Fv}$
by \Cref{P:Pfisterresidue},
so under the component map $I_q^d(Fv) \to W_1 \widehat{F_v}$ of $\iota$ we obtain the class of the form $(-a_d) \llangle \overline{a_1}, \dotsc, \overline{a_{d-1}}, \overline{b}]]_{\widehat{F}_v}$ for any lifting $Fv \to \widehat{F_v}$.
This is identical to the class of $[q'] - [a_d q']$ in $W_1 \widehat{F_v}$,
since over $\widehat{F_v}$ both $q'$ and $\llangle \overline{a_1}, \dotsc, \overline{a_{d-1}}, \overline{b}]]_{\widehat{F_v}}$ are
anisotropic unimodular forms with the same first residue form $\llangle \overline{a_1}, \dotsc, \overline{a_{d-1}}, \overline{b}]]_{Fv}$
and therefore isometric,
and the class of $q'$ vanishes in $W_1 \widehat{F_v}$ by definition.
This proves that the left-hand square commutes.

To prove the that the right-hand square of \eqref{eq:diagram-two-MilnorES} commutes, let us take $v \in \Val(F/K)$.
We need to consider the definition of the transfer map $s_v^{AJ}$ from \autocite[Definition~4.3]{Aravire-Jacob}.
In general this is quite complicated.
However, for an element of $W_1 \widehat{F_v}$ which can be represented as $[\pi q]$ for some $[q] \in I_q(Fv)$
(again considered under the embedding $I_q(Fv) \to I_q \widehat{F_v}$) and a certain uniformiser $\pi$ of $v$,
the element $s_v^{AJ}([\pi q]) \in I_qK$ is by definition simply the usual transfer of $[q]$ with respect to
a certain $K$-linear functional $Fv \to K$.
Given our definition of $\iota$, this suffices to ensure that the right-hand square commutes.
\end{proof}

The following is a basic lemma on function fields;
we give a proof for lack of a reference.
\begin{lem}\label{lem:smoothify-function-field}
  Let $F/K$ be a regular function field in one variable.
  Then there exists a purely inseparable finite extension $K'/K$
  such that $FK'/K'$ is the function field of a smooth projective curve $C/K'$.
  In particular, for every finite extension $L/K'$, $FL/L$
  is the function field of the base-changed smooth projective curve $C_L/L$.
\end{lem}
\begin{proof}
  For any field extension $K'/K$, we may consider the genus $g(FK'/K')$,
  i.e.\ the non-negative integer making the Riemann--Roch Theorem true;
  equivalently, this is the genus of the unique regular projective curve over $K'$
  of which $FK'$ is the function field.
  Pick a finite purely inseparable extension $K'/K$ with $g(FK'/K')$ minimal
  among all such choices.
  Let $C$ be the unique regular projective curve over $K'$ with function field $FK'$.

  We now use some general principles for the genus \autocite[Theorem 2.5.1]{Poonen_RationalPoints}.
  For any finite extension $L/K'$, we have $g(FL/L) \leq g(FK'/K')$,
  and the genus does not change under separable extensions.
  By choice of $K'$, this means that $g(FL/L) = g(FK'/K')$ for all $L/K'$.
  Equivalently, the base change $C_L$ remains regular.
  Therefore the curve $C$ is geometrically regular,
  which is equivalent to being smooth \cite[Proposition 3.5.22]{Poonen_RationalPoints}.
\end{proof}

The following is a technical result to be used in reduction steps in the proof of \Cref{T:Reciprocity-Main}.
\begin{lem}\label{lem:reciprocity-normal-reduction}
  In the situation of \Cref{T:Reciprocity-Main}, consider $\alpha \in \bigoplus_{v \in \Val(F/K)} I_q^d(Fv)$
  which maps to $0$ in $I_q^d K$.
  Let $K'/K$ be a finite normal extension, set $F' = FK'$, and consider the following commutative diagram:
  \[ \xymatrix{
      \bigoplus_{v \in \Val(F/K)} I_q^d(Fv) \ar[r] & I_q^d K \\
      \bigoplus_{w \in \Val(F'/K')} I_q^d(F'w) \ar[u] \ar[r] & I_q^d K' \ar[u]
    }\]
  Here all arrows are (sums of) transfer maps;
  the left-hand upward arrow is the sum of maps $I_q^d(F'w) \to I_q^d(Fv)$,
  which are zero unless $w$ lies over $v$, and otherwise are the transfer for $Fv \subseteq F'w$.
  There exists $\alpha' \in \bigoplus_{w \in \Val(F'/K')} I_q^d (F'w)$ mapping to $0$ in $I_q^d K'$
  and mapping to $\alpha$ under the vertical map,
  such that every place $w$ contributing a non-zero term to $\alpha'$
  lies above some place $v$ contributing a non-zero term to $\alpha$.
\end{lem}
\begin{proof}
  If $\alpha = 0$, we can simply choose $\alpha' = 0$.
  Therefore assume $\alpha \neq 0$.
  Enumerate the $v$ contributing non-trivially to $\alpha$ as $v_1, \dotsc, v_n$.
  For each of these $v_i$, pick a $w_i \in \Val(F'/K')$ lying above $v_i$.
  By surjectivity of the individual transfer maps $I_q^d(F'w_i) \to I_q^d(Fv_i)$ (\Cref{lem:transfer-surj-kernel}),
  we can pick $\beta \in \bigoplus_{w \in \Val(F'/K')} I_q^d(F'w)$ mapping to $\alpha$,
  and such that the only non-trivial contributions to $\beta$ come from the $w_i$.
  However, the image of $\beta$ in $I_q^d K'$ will not usually be $0$;
  in any case, it must be in the kernel of the transfer $I_q^d K' \to I_q^d K$
  by commutativity of the diagram.

  Observe that $\operatorname{Aut}(K'/K)$ acts on the bottom row of the diagram
  (where in particular the places $w$ are permuted by this action),
  and the entire diagram is $\operatorname{Aut}(K'/K)$-invariant when
  we let this group act trivially on the top row.
  According to \Cref{lem:transfer-surj-kernel},
  we may write the image of $\beta$ in $I_q^d K'$ as $\sum_{j = 1}^m (\sigma_j q_j - q_j)$
  with $\sigma_j \in \operatorname{Aut}(K'/K)$, $q_j \in I_q^d K'$.
  For each $q_j$, choose $p_j \in I_q^d(F'w_1)$ mapping to $q_j$ under the transfer
  (where the choice of the valuation $w_1$ is arbitrary among the $w_i$).
  Set $\alpha' = \beta - \sum_{j=1}^m (\sigma_j p_j - p_j)$, where we see each $p_j$ as an element
  of $\bigoplus_{w \in \Val(F'/K')} I_q^d(F'w)$.
  Then the image of $\alpha'$ in $I_q^d K'$ is zero by construction, since the image
  of $\beta$ is precisely cancelled out by the additional contribution.
  On the other hand, the image of each $\sigma_j p_j - p_j$ in $\bigoplus_{v \in \Val(F/K)} I_q^d(Fv)$
  is $0$, and so $\alpha'$ has the same image in this group as $\beta$,
  i.e.\ $\alpha'$ maps to $\alpha$.
\end{proof}

\begin{proof}[{Proof of \Cref{T:Reciprocity-Main}}]
  Let us first show that the given maps form a complex, i.e.\ their composition is the zero map.
  Choose $T \in F$ transcendental over $K$, so that $F/K(T)$ is a finite extension.
  We have the following diagram:
  \[ \xymatrix{
      I_q^{d+1}F \ar[r] \ar[d] & \bigoplus_{v \in \Val(F/K)}  I_q^d(Fv)   \ar[r] \ar[d] & I_q^d K \ar@{=}[d] \\
      I_q^{d+1} (K(T)) \ar[r] & \bigoplus_{v' \in \Val(K(T)/K)} I_q^d(K(T)v') \ar[r] & I_q^d K
    } \]
  Here the vertical maps are given by transfers maps, and the diagram commutes:
  This is by \Cref{lem:residue-transfer-commute} for the left square
  (using that the fundamental equality always holds for extensions of function fields
  \cite[Theorem 5.1.14]{VillaSalvador}),
  and functoriality of transfer for the right.
  The bottom row is a complex by \Cref{lem:Milnor-seq-char-2} in characteristic $2$ and
  \autocite[Corollary 21.7]{ElmanKarpenkoMerkurjev} otherwise.
  It follows that the top row is also a complex.

  Let us now show exactness.
  Consider $\alpha \in \bigoplus_{v \in \Val(F/K)} I_q^d(Fv)$ mapping to $0$ in $I_q^d K$.
  We must show that $\alpha$ lies in the image of $I_q^{d+1}F$.
  Choose a finite normal extension $K'/K$ such that, setting $F' = FK'$, every $w \in \Val(F'/K')$
  lying over one of the $v$ contributing to $\alpha$ has residue field $K'$,
  and such that $F'/K'$ is the function field of a smooth projective curve;
  this is possible by \Cref{lem:smoothify-function-field}.

  We apply \Cref{lem:reciprocity-normal-reduction} to obtain $\alpha' \in \bigoplus_{w \in \Val(F'/K')} I_q^d (F'w)$
  which maps to $0$ in $I_q^d K'$ and to $\alpha$ under the vertical map,
  and such that the only $w$ contributing to $\alpha'$ lie over some $v$ contributing to $\alpha$.
  Enumerate the $w$ contributing to $\alpha'$ as $w_1, \dotsc, w_n$.
  By construction one has $F'w_j = K'$ for $j = 1, \ldots, n$, and thus for any field extension $K''/K'$, each $w_j$ extends uniquely to a valuation in $\Val(F'K''/K'')$.
  If $\alpha'$ is zero, then so is $\alpha$, and there is nothing to show;
  therefore assume that $n \geq 1$.

  Choose a finite normal extension $K''/K'$ such that, with $F'' = F'K''$,
  the class of the divisor $u_i - u_1$ in $\Pic^0(F''/K'')$
  (the group of divisors of degree $0$ modulo principal divisors)
  is divisible by $2$ for every $i = 1, \dotsc, n$,
  where we write $u_i$ for the unique extension of $w_i$ to $F''/K''$.
  (To see that this exists, recall that for every $K''/K'$, the group $\Pic^0(F'K''/K'')$
  is identified with the group of $K''$-rational points of the Jacobian variety $J/K'$ of
  the smooth projective curve $C/K'$ defining $F'/K'$, which is an abelian variety
  \cite[Theorem 5.7.6]{Poonen_RationalPoints};
  since the group of rational points $J(\overline{K'})$ over the algebraic closure $\overline{K'}$ is $2$-divisible
  (as multiplication by $2$ is an isogeny $J \to J$),
  we can find $x \in J(\overline{K'})$
  with $2x = [w_i - w_1]$, where $[w_i - w_1] \in \Pic^0(F'/K') = J(K')$ is the class of the divisor $w_i - w_1$,
  and take $K''$ large enough so that $x \in J(K'')$.)

  We now have the following commutative diagram of transfer and residue maps:
  \[ \xymatrix{
      I_q^{d+1} F          \ar[r] & \displaystyle\bigoplus_{v \in \Val(F/K)}     I_q^d(Fv)   \ar[r] &        I_q^d K \\
      I_q^{d+1} F'  \ar[u] \ar[r] & \displaystyle\bigoplus_{w \in \Val(F'/K')}   I_q^d(F'w)  \ar[u] \ar[r] & I_q^d K' \ar[u] \\
      I_q^{d+1} F'' \ar[u] \ar[r] & \displaystyle\bigoplus_{u \in \Val(F''/K'')} I_q^d(F''u) \ar[u] \ar[r] & I_q^d K'' \ar[u]
  }\]
  Using \Cref{lem:reciprocity-normal-reduction} a second time,
  we obtain $\alpha'' \in \bigoplus_{u \in \Val(F''/K'')} I_q^d(F''u)$
  mapping to $0$ in $I_q^d K''$ and to $\alpha'$ under the vertical map,
  and such that only the $u_i$ contribute to $\alpha''$.
  We therefore have $\alpha_i \in I_q^d K'' = I_q^d(F''u_i)$ for $i = 1, \dotsc, n$
  such that $\alpha'' \in \bigoplus_{u \in \Val(F''/K'')} I_q^d(F''u)$ is the sum of the $\alpha_i$
  (each $\alpha_i$ contributing to the term at $u_i$).

  By the choice of $K''$, for each $i > 1$ there exists $f_i \in F''^\times$
  such that the divisor $(f_i)$ is equal to the divisor $u_i - u_1 + 2 D_i$ for some divisor $D_i$.
  For a $d$-fold Pfister form $q = \llangle a_1, \dotsc, a_{d-1}, b]]_{K''}$,
  it follows from \Cref{P:Pfisterresidue} that the form $q' = \llangle a_1, \dotsc, a_{d-1}, f_i, b]]_{F''}$
  has $\partial_u [q'] = 0$ in $I_q^d(Fu)$ at all $u \in \Val(F''/K'')$ except for $u = u_i$ and $u = u_1$,
  where $\partial_u [q'] = [q]$.
  By writing each $\alpha_i$ as a sum of classes of $d$-fold Pfister forms and applying this construction to each summand,
  we obtain in this way a class $\beta_i \in I_q^{d+1}F''$ such that 
  $\partial_u \beta_i = 0$ at all $u \in \Val(F''/K'')$ except for $u = u_i$ and $u = u_1$,
  where $\partial_u \beta_i = \alpha_i$.
  Set $\beta = \beta_2 + \dotsb + \beta_n \in I_q^{d+1}F''$.
  Then $\beta$ has residue $\alpha_i$ at $u_i$ for $i = 2, \dotsc, n$,
  residue $\alpha_2 + \dotsb + \alpha_n$ at $u_1$,
  and residue zero everywhere else.
  Since the image of $\alpha''$ in $I_q^d K''$ is zero,
  we have $\alpha_1 + \alpha_2 + \dotsb + \alpha_n = 0$ in $I_q^d K''$,
  and so $\alpha_1 = \alpha_2 + \dotsb + \alpha_n$ since $I_q^d K'' = I_q^d K'' / I_q^{d+1} K''$ is $2$-torsion.
  Therefore the image of $\beta$ under the sum of residue maps is precisely $\alpha''$.
  Using commutativity of the diagram, the image of $\beta$ in $I_q^{d+1} F$ maps to $\alpha$.
\end{proof}

\section{Describing rings of $S$-integers}\label{sect:UniversalDefinitionGeneralTechnique}
We will now first explain a general technique to construct universal first-order definitions of certain subrings in a field, which we will then in the next section combine with definability results from \cite{Andromeda-1} to obtain universal definitions of rings of $S$-integers in function fields over several types of base fields.

While most of the results in this section are purely algebraic, we gather here for the reader's convenience the notations and concepts we will borrow from first-order logic.
For a general reference, see \cite{Marker}, starting with Example 1.2.8.
We shall work in the signature of rings $\Lar$, consisting of binary operation symbols $+, -, \cdot$, and constant symbols $0$ and $1$.
We will consider rings (in particular fields) as structures in the language $\Lar$ by interpreting the operation and constant symbols in the natural way.
For a ring $R$, we will denote by $\Lar(R)$ the signature obtained by adding to $\Lar$ a constant symbol for every element of $R$.
Then $R$ (and more generally, every ring $R' \supseteq R$) is an $\Lar(R)$-structure in the natural way.

For a ring $R$ and $n \in \nat$, we will say that a subset $S \subseteq R^n$ is \emph{existentially definable in $R$} if there exists an existential $\Lar(R)$-formula defining $S$ (in the $\Lar(R)$-structure $R$),
and \emph{universally definable in $R$} if there exists a universal $\Lar(R)$-formula defining $S$ in $R$.
If the ring $R$ is clear from context, we simply call $S$ existentially (universally) definable.
We abbreviate existentially (universally) definable to $\exists$-definable ($\forall$-definable).

The more algebraically inclined reader may alternatively take the following equivalent definition when $R$ is a non-algebraically closed field, which can be phrased without any reference to first-order logic:
a subset $S \subseteq R^n$ is $\exists$-definable in $R$ if there exist $m \in \nat$ and a polynomial $f \in R[X_1, \ldots, X_n, Y_1, \ldots, Y_m]$ such that
$$ S = \lbrace (x_1, \ldots, x_n) \in R^n \mid \exists (y_1, \ldots, y_m) \in R^m : f(x_1, \ldots, x_n, y_1, \ldots, y_m) = 0 \rbrace, $$
and $S$ is $\forall$-definable if $R^n \setminus S$ is $\exists$-definable.
The equivalence of these two definitions when $R$ is a non-algebraically closed field is well-known, see e.g.~\cite[Proposition 2.1]{DaansGlobal}.
We also use the word \emph{diophantine} for existentially definable; while some authors make some technical distinctions, we will use the terms interchangeably.
For subsets of $R^1$, the definition of existentially definable given in the introduction for an arbitrary field $R$ agrees with the one here.

Finally, in the statement of \Cref{ce-sets-ea-definable} we will also mention \emph{existential-universally definable} ($\exists\forall$-definable) subsets, i.e.~those which can be defined by a existential-universal $\Lar(R)$-formula.
In algebraic terms, for a non-algebraically closed field $R$, a subset $S \subseteq R^n$ is $\exists\forall$-definable in $R$ if there exist $m_1, m_2 \in \nat$ and a polynomial $f \in R[X_1, \ldots, X_n, Y_1, \ldots, Y_{m_1}, Z_1, \ldots, Z_{m_2}]$ such that
$$ S = \lbrace x \in R^n \mid \exists y \in R^{m_1} :  \forall z \in R^{m_2} : f(x, y, z) \neq 0 \rbrace.$$

We now start to describe our general technique.
Recall from \Cref{C:residueHomomorphism} that for a $\zz$-valued field $(F, v)$ and $d \in \nat^+$, assuming when $v$ is dyadic that $[Fv : (Fv)\pow{2}] < 2^d$, there is a unique homomorphism $$\partial_v : I^{d+1}_q(F)/I^{d+2}_q(F) \to I^{d}_q(Fv)/I^{d+1}_q(Fv)$$ which maps the class of a $(d+1)$-fold Pfister form $q$ which is anisotropic over $F_v$ to the class of its residue form $\RF{1}{q_{F_v}}$.
We write $\partial_v q$ instead of $\partial_v [q]$.

Given a set of $\zz$-valuations $\mc{V}$ on $F$ such that $[Fv : (Fv)\pow{2}] < 2^d$ for all dyadic $v \in \mc{V}$ and given $\alpha \in I^{d+1}_q(F)/I^{d+2}_q(F)$ (and in particular, for a $(d+1)$-fold Pfister form over $F$), we denote by $\delta \alpha \subseteq \mc{V}$ the set of valuations $v \in \mc{V}$ for which $\partial_v \alpha \neq 0$.
In all of the applications we will study in \Cref{sect:universalFunctionfield}, the set $\delta q$ also has the following alternative description in terms of anisotropy over the henselisations $F_v$ of $F$ with respect to the valuations $v \in \mc{V}$.
\begin{prop}\label{P:deltavsDelta}
Let $F$ be a field, $\mc{V}$ a set of $\zz$-valuations on $F$, $d \in \nat^+$ such that $I^{d+1}_q(Fv) = 0$ for all $v \in \mc{V}$, and $[Fv : (Fv)\pow{2}]<2^d$ for all dyadic $v \in \mc{V}$.
For any $(d+1)$-fold Pfister form $q$ over $F$ we have
$$ \delta q = \lbrace v \in \mc{V} \mid q \text{ is anisotropic over } F_v \rbrace.$$
\end{prop}
\begin{proof}
This is immediate from the fact that $\partial_v : I^{d+1}_q(F_v) \to I^d_q(Fv)$ is an isomorphism for each $v \in \mc{V}$, see \Cref{C:residueHomomorphism}.
\end{proof}
\begin{defi}\label{D:finiteSupport}
Let $F$ be a field. A set of valuations $\mc{V}$ on $F$ is said to satisfy the \emph{finite support property}\index{finite support property} if for all $a \in F^\times$, the set of valuations $v$ in $\mc{V}$ for which $v(a) \neq 0$ is finite.
\end{defi}
\begin{vbn}\label{E:finiteSupport}
(1) If $F$ is a number field, then the set of all $\zz$-valuations on $F$ has the finite support property.

(2) If $K$ is a field and $F/K$ a function field in one variable, then the set $\mc{V}(F/K)$ of $\zz$-valuations on $F$ trivial on $K$ has the finite support property.

From (1) and (2) it follows that the set of all $\zz$-valuations on any global field has the finite support property.
\end{vbn}
\begin{opm}\label{R:finiteSupportQF}
If $\mc{V}$ is a set of valuations on $F$ with the finite support property, then for every $d \in \nat^+$ and $\alpha \in I^{d+1}_q(F)/I^{d+2}_q(F)$ the set $\delta \alpha$ is finite.
To see this, it suffices to consider the case where $\alpha$ is the class of a $(d+1)$-fold Pfister form $\llangle a_1, \ldots, a_{d}, b]]_F$ over $F$ with $b \neq 0$. 
For all but finitely many valuations $v \in \mc{V}$ we have $v(a_1) = \ldots = v(a_{d}) = v(b) = v(1+4b) = 0$,
and then $\partial_v \alpha = 0$ by \Cref{P:Pfisterresidue}.
\end{opm}

For the rest of this section we let $F$ be a field, $d \in \nat^+$, $\mc{V}$ a set of $\zz$-valuations on $F$ with the finite support property and such that $[Fv : (Fv)\pow{2}] < 2^d$ for all dyadic $v \in \mc{V}$, and $S \subseteq \mathcal{V}$ a finite set. 
The goal of this section is to develop a general technique, abstracting from Koenigsmann's proof in \cite{Koe16},
which we will use to establish universal definability of the ring of $S$-integers $\mc{O}_S = \bigcap_{v \in \mc{V} \setminus S} \mc{O}_v$ in $F$.
We will in \Cref{sect:universalFunctionfield} apply this to the case where $F$ is a function field in one variable over a field $K$ and $\mc{V} = \mc{V}(F/K)$.

\begin{defi}
For a quadratic form $q$ defined over $F$ and $c \in F^\times$, we define the following sets:
\begin{align*}
\Odd(c) &= \lbrace v \in \mc{V} \mid v(c) \text{ odd} \rbrace \\
\Neg(c) &= \lbrace v \in \mc{V} \mid v(c) < 0 \rbrace \\
J_c(q) &= \bigcap \lbrace \mathfrak{m}_v \mid v \in \delta q \cap \Odd(c) \rbrace \\
H_c(q) &= \bigcap \lbrace \mathfrak{m}_v^{-v(c)} \mid v \in \delta q \cap \Neg(c) \rbrace
\end{align*}
\end{defi}
Here and throughout this section, we take the convention that an intersection of subsets of a field $F$ indexed by an empty set is equal to all of $F$.
That is, $J_c(q) = F$ if $\delta q \cap \Odd(c) = \emptyset$, and $H_c(q) = F$ if $\delta q \cap \Neg(c) = \emptyset$.

The sets $J_c(q)$ were essentially introduced in \cite{Koe16},
and the sets $H_c(q)$ were introduced in \cite{DaansGlobal} with the goal of studying the characteristic $2$ case.
For $m, p \in \nat$ and $a_1, \ldots, a_m, b_1, \ldots, b_p \in F^\times$, we will write 
$(J_{a_1} \cap \ldots \cap J_{a_m} \cap H_{b_1} \cap \ldots \cap H_{b_p})(q)$
instead of
$J_{a_1}(q) \cap \ldots \cap J_{a_m}(q) \cap H_{b_1}(q) \cap \ldots \cap H_{b_p}(q).$
\begin{lem}\label{L:PfisterNontrivialSlot}
Consider a $(d+1)$-fold Pfister form $q = \llangle a_1, \ldots, a_d, b ]]_F$ for some elements $a_1, \ldots, a_d, b \in F$ and let $v \in \delta q$. If $v$ is non-dyadic, then $v \in \bigcup_{i=1}^d \Odd(a_i) \cup \Odd(1+4b)$.
If $v$ is dyadic, then $v \in \bigcup_{i=1}^d \Odd(a_i) \cup \Neg(b)$.
\end{lem}
\begin{proof}
Suppose that $v(a_i)$ is even for $1 \leq i \leq d$.
We need to show that $v(1+4b)$ is odd when $v$ is non-dyadic, and that $v(b) < 0$ if $v$ is dyadic.
To this aim, by multiplying each of the $a_i$ by a square in $F^\times$, we may assume without loss of generality that $a_1, \ldots, a_d \in \mc{O}_v^\times$.
Furthermore, if we had $v(b) > 0$, then $\llangle b]]_{F_v}$ would be isotropic and hence $q_{F_v}$ would be isotropic, whereby $\partial_v q = 0$ in $I_q^d(Fv) / I_q^{d+1}(Fv)$.
Since this contradicts the assumption that $v \in \delta q$, we infer that $v(b) \leq 0$.

First consider the case where $v$ is non-dyadic.
Assume that $v(1+4b)$ is even, so there exists $u \in F^\times$ with $v(u^2) = 2v(u) = -v(1+4b)$.
One computes that
$$ X^2 - XY - bY^2 = (X + Y\frac{1-u}{2u})^2 - (X + Y\frac{1-u}{2u})(Y\frac{1}{u}) - (u^2b + \frac{u^2 - 1}{4})(Y\frac{1}{u})^2 $$
and thus that $\llangle b ]]_F \cong \llangle u^2b + \frac{u^2 - 1}{4} ]]_F$.
Hence, after replacing $b$ by $u^2b + \frac{u^2 - 1}{4}$, we may assume that $v(1+4b) = 0$, whereby automatically $v(b) = 0$.
But then it follows from \Cref{P:Pfisterresidue} that $\partial_v q = 0$ in $I_q^d(Fv) / I_q^{d+1}(Fv)$,
contradicting the assumption that $v \in \delta q$. We conclude that $v(1+4b)$ is odd.

Now assume that $v$ is dyadic.
If $v(b) = 0$, then it again follows from \Cref{P:Pfisterresidue} that $\partial_v q = 0$ in $I_q^d(Fv) / I_q^{d+1}(Fv)$,
contradicting the assumption that $v \in \delta q$.
Since we already know that $v(b) \leq 0$, we conclude that $v(b) < 0$.
\end{proof}
We now introduce a property of a set $X \subseteq F^{d+1}$ which we will need for our general statement.
\begin{defi}\label{D:PfisterStratum}
Let $F$ be a field, $d \in \nat^+$, $\mc{V}$ a set of $\zz$-valuations on $F$ with the finite support property and such that $[Fv : (Fv)\pow{2}] < 2^d$ for all dyadic $v \in \mc{V}$.
For a finite set $S \subseteq \mc{V}$, we call a set $X \subseteq F^{d+1}$ a \emph{Pfister stratum for $S$ (with respect to $\mc{V}$)}\index{Pfister stratum} if the following two properties are satisfied.
\begin{enumerate}[(i)]
\item\label{it:propA} For any $(x_1, \ldots, x_d, y) \in X$ we have $x_1, \ldots, x_d, y, 1+4y \neq 0$, $\delta\llangle x_1, \ldots, x_d , y ]]_F \nsubseteq S$, and
$$ S \cap \Big( \bigcup_{i=1}^{d-1} \Odd(x_i) \cup \Odd(1+4y) \cup \Neg(y) \Big) = \emptyset. $$
\item\label{it:propB} For any $w \in \mc{V} \setminus S$, there exists $(x_1, \ldots, x_d, y) \in X$ with $y \in \mathcal{O}_w$ and $\delta \llangle x_1, \ldots, x_d, y ]]_F \subseteq S \cup \lbrace w \rbrace$.
\end{enumerate}
\end{defi}
Our first aim now is to show how $\exists$-definable Pfister strata yield universal definitions of rings of $S$-integers (\Cref{T:UniversalDefinitionsGeneralTechnique} below).
We will discuss instances where such Pfister strata can be found at the end of this section.

For a set $X \subseteq F^{d+1}$ and $\pi \in F$, we will throughout this section use the following notation:
$$X_\pi = \lbrace (a_1, \ldots, a_{d}, b) \in F^{d+1} \mid (a_1, \ldots, a_{d-1}, a_d\pi, b) \in X \rbrace.$$
\begin{lem}\label{L:EtoAMainLem}
Let $\pi \in F^\times$ be such that $S \subseteq \Odd(\pi)$.
Let $X \subseteq F^{d+1}$ be a Pfister stratum for $S$, and let $a_0 \in F^\times$ be such that $\Odd(\pi) \setminus S \subseteq \Odd(a_0) \subseteq \mc{V} \setminus S$.
We have 
$$ \bigcup_{v \in \mc{V} \setminus S} \mathfrak{m}_v = \bigcup_{(a_1, \ldots, a_d, b) \in X_\pi} \Big(\bigcap_{i=0}^{d} J_{a_i} \cap J_{1+4b} \cap H_b\Big)(\llangle a_1, \ldots, a_{d-1}, a_d\pi, b]]_F).$$
\end{lem}
\begin{proof}
We first show the inclusion from left to right.
Let $w \in \mc{V} \setminus S$.
Since $X$ is a Pfister stratum for $S$, we can by property \eqref{it:propB} find $(a_1, \ldots, a_d, b) \in X_\pi$ such that, for $q = \llangle a_1, \ldots, a_{d-1}, a_d\pi, b ]]_F$, one has $\delta q \subseteq S \cup \lbrace w \rbrace$ and $b \in \mathcal{O}_w$.
Observe that, for $v \in \delta q \setminus \lbrace w \rbrace$, one has $v \not\in \bigcup_{i=0}^{d-1} \Odd(a_i) \cup \Odd(1+4b) \cup \Neg(b)$ by \eqref{it:propA} and the choice of $a_0$.
By \Cref{L:PfisterNontrivialSlot} we must have $v \in \Odd(a_d\pi)$, and since $v \in \Odd(\pi)$, $v \not\in \Odd(a_d)$.
Furthermore, $w \not\in \Neg(b)$.

We obtain that 
$$\delta q \cap \Neg(b) = \emptyset \enspace\text{and}\enspace \delta q \cap \Big( \bigcup_{i=0}^{d} \Odd(a_i) \cup \Odd(1+4b) \Big) \subseteq \lbrace w \rbrace$$
whereby $H_b(q) = F$ and thus
\begin{align*}
&\Big(\bigcap_{i=0}^{d} J_{a_i} \cap J_{1+4b} \cap H_b\Big)(q) = \Big(\bigcap_{i=0}^{d} J_{a_i} \cap J_{1+4b}\Big)(q) \\
= \enspace &\bigcap \Big\lbrace \mathfrak{m}_v \enspace\Big|\enspace v \in \delta q \cap \Big( \bigcup_{i=0}^{d} \Odd(a_i) \cup \Odd(1+4b) \Big) \Big\rbrace
\supseteq \enspace \mathfrak{m}_w
\end{align*}
concluding the proof of the first inclusion.

We now prove the inclusion from right to left. Let $(a_1, \ldots, a_d, b) \in X_\pi$ and set $q = \llangle a_1, \ldots, a_{d-1}, a_d\pi, b]]_F$.
By \eqref{it:propA} we have $\delta q \setminus S \neq \emptyset$. Let $w \in \delta q \setminus S$. The required statement now follows from \Cref{L:PfisterNontrivialSlot}, using that, if $w \in \Odd(a_d\pi)$, then either $w \in \Odd(a_d)$ or $w \in \Odd(\pi) \setminus S \subseteq \Odd(a_0)$.
\end{proof}
We will now show that the sets $J_c(q)$ and $H_c(q)$ are existentially definable, uniformly in the parameters defining $q$, under the condition that the sets $\bigcap_{v \in \delta q} \mc{O}_v$ are existentially definable, uniformly in the parameters defining $q$ (\Cref{IJH}).
To be precise: when we are given $n, m \in \nat$, a subset $A \subseteq F^n$, and for each $a \in A$ a set $S_a \subseteq F^m$, we say that the sets $S_a$ are existentially definable in $F$, \emph{uniformly in the parameters $a$} if the set
$$ \lbrace (a, x) \in F^{n} \times F^{m} \mid a \in A, x \in S_a \rbrace $$
is $\exists$-definable in $F$.\footnote{%
  In all cases of interest to us, the parameter set $A$ will be the set of tuples $(a_1, \dotsc, a_d, b)$
  which define a valid Pfister form $\llangle a_1, \dotsc, a_d, b]]_F$ for some $d$,
  so $A$ is simply given by the conditions $a_1, \dotsc, a_d, 1+4b \neq 0$, and therefore $A$ itself will trivially be quantifier-freely definable.
  In more general circumstances, where $A$ is not guaranteed to be definable at all,
  other authors might prefer a less strict definition of uniform existential definability.
}
\begin{lem}\label{jacoblem}
Let $F$ be a field, $S$ a finite set of $\zz$-valuations on $F$, and $c \in F^\times$. Let $R = \bigcap_{v \in S} \mathcal{O}_v$.
We have
\begin{align*}
F\pow{2} \cdot R^\times &= \bigcap_{v \in S} v^{-1}(2\zz), \\
(c \cdot F\pow{2} \cap (1 - F\pow{2} \cdot R^\times))\cdot R &= \bigcap \lbrace\mathfrak{m}_v \mid v \in S \cap \Odd(c) \rbrace,  \\
(c^{-1} \cdot R + c \cdot R^{-1})^{-1} \cup \lbrace 0 \rbrace &= \bigcap\lbrace\mathfrak{m}_v^{-v(c)} \mid v \in S \cap \Neg(c) \rbrace. 
\end{align*}
Here the operations are interpreted elementwise in the natural way,
so $A^{-1} = \{ x^{-1} \mid x \in A \setminus \{ 0 \} \}$,
$A + B = \{ a + b \mid a \in A, b \in B \}$ and so forth for subsets $A, B \subseteq F$.
\end{lem}
\begin{proof}
See \autocite[Lemma 5.4]{DaansGlobal} (following \autocite[Section 2, Step 3]{Koe16}).
\end{proof}
\begin{lem}\label{IJH}
Assume that, for any $(d+1)$-fold Pfister form $q = \llangle a_1, \ldots, a_d, b ]]_F$ the rings $\bigcap_{v \in \delta q} \mathcal{O}_v$ are existentially definable in $F$, uniformly in the parameters $a_1, \ldots, a_d, b$.
Then the sets $J_c(\llangle a_1, \ldots, a_d, b]]_F)$ and $H_c(\llangle a_1, \ldots, a_d, b]]_F)$ are existentially definable in $F$, uniformly in the parameters $a_1, \ldots, a_d, b, c$.
\end{lem}
\begin{proof}
For $a_1, \ldots, a_d, b \in F$ with $a_1, \ldots, a_d, 1+4b \neq 0$ we set $$R_{a_1, \ldots, a_d, b} = \bigcap_{v \in \delta \llangle a_1, \ldots, a_d, b]]_F} \mc{O}_v.$$
By \Cref{jacoblem} we have
\begin{align*}
J_c(\llangle a_1, \ldots, a_d, b]]_F) &= (c \cdot F\pow{2} \cap (1 - F\pow{2} \cdot R_{a_1, \ldots, a_d, b}^\times))\cdot R_{a_1, \ldots, a_d, b} \enspace\text{and} \\
H_c(\llangle a_1, \ldots, a_d, b]]_F) &= (c^{-1} \cdot R_{a_1, \ldots, a_d, b} + c \cdot R_{a_1, \ldots, a_d, b}^{-1})^{-1} \cup \lbrace 0 \rbrace.
\end{align*}
Since the sets $R_{a_1, \ldots, a_d, b}$ are uniformly $\exists$-definable in $F$ by assumption, the statement follows.
\end{proof}
\begin{lem}\label{P:EtoAgeneral}
Let $V$ be a non-empty set of valuations on $F$.
The set $\bigcup_{v \in V} \mathfrak{m}_v$ is $\exists$-definable in $F$ if and only if $\bigcap_{v \in V} \mathcal{O}_v$ is $\forall$-definable in $F$.
\end{lem}
\begin{proof}
As in \autocite[Proposition 6.1]{DaansGlobal}, this follows from the observation
\begin{displaymath}
\bigcap_{v \in V} \mathcal{O}_v = \Big(F \setminus \Big(\bigcup_{v \in V} \mathfrak{m}_v\Big)^{-1}\Big) \cup \lbrace 0 \rbrace. \qedhere
\end{displaymath}
\end{proof}
\begin{stel}\label{T:UniversalDefinitionsGeneralTechnique}
Suppose that $X \subseteq F^{d+1}$ is a Pfister stratum for $S$, and that $X$ is $\exists$-definable in $F$.
Furthermore, assume that for $(d+1)$-fold Pfister forms $q = \llangle a_1, \ldots, a_d, b ]]_F$ the rings $\bigcap_{v \in \delta q} \mathcal{O}_v$ are $\exists$-definable, uniformly in the parameters $a_1, \ldots, a_d, b$.
Then $\bigcup_{v \in \mc{V} \setminus S} \mathfrak{m}_v$ is $\exists$-definable in $F$, and $\mc{O}_S = \bigcap_{v \in \mc{V} \setminus S} \mathcal{O}_v$ is $\forall$-definable in $F$.
\end{stel}
\begin{proof}
In view of \Cref{P:EtoAgeneral} we only need to show that $\bigcup_{v \in \mc{V} \setminus S} \mathfrak{m}_v$ is $\exists$-definable in $F$.

By Weak Approximation, we can find $\pi \in F^\times$ such that $v(\pi) = 1$ for all $v \in S$, in particular $S \subseteq \Odd(\pi)$. Applying Weak Approximation again, we can find $a_0 \in F^\times$ such that $v(a_0) = 1$ if $v \in \Odd(\pi) \setminus S$ and $v(a_0) = 0$ if $v \in S$.
By \Cref{L:EtoAMainLem} we have
$$ \bigcup_{v \in \mc{V} \setminus S} \mathfrak{m}_v = \bigcup_{(a_1, \ldots, a_d, b) \in X_\pi} \Big(\bigcap_{i=0}^{d} J_{a_i} \cap J_{1+4b} \cap H_b\Big)(\llangle a_1, \ldots, a_{d-1}, a_d\pi, b]]_F).$$
This set is $\exists$-definable by the hypotheses and \Cref{IJH}.
\end{proof}

We have now discussed an abstract technique to obtain universal definability of rings of $S$-integers in fields.
This technique relies on finding an $\exists$-definable Pfister stratum $X \subseteq F^{d+1}$.
We discuss in the remainder of this section how such a set can be obtained in certain cases.

Recall the standing hypotheses of the section: we let $F$ be a field, fix $d \in \nat^+$, and we let $\mc{V}$ be a set of $\zz$-valuations with the finite support property and such that $[Fv : (Fv)\pow{2}] < 2^d$ for all dyadic $v \in \mc{V}$.

Let $S \subseteq \Val$ be finite, and for each $v \in \Val$, let $q_v$ be a $d$-fold Pfister form over $Fv$.
To the finite tuple $(q_v)_{v \in S}$ we can associate a subset of $F^{d+1}$:
$$\Phi_{(q_v)_{v \in S}} = \left\lbrace (a_1, \ldots, a_{d}, b) \in F^{d+1} \enspace\middle|\enspace \forall v \in S : \begin{array}{l}
a_1, \ldots, a_{d-1}, b, 1+4b \in \mathcal{O}_v^\times, \\
\partial_v \llangle a_1, \ldots, a_{d}, b ]]_F = [q_v] + I_q^{d+1} (Fv)
\end{array}
\right\rbrace.$$
\begin{lem}\label{Phidefinable}
Let $S \subseteq \mc{V}$ be finite.
For each $v \in S$, let $q_v$ be an anisotropic $d$-fold Pfister form over $Fv$.
Assume that $\mathcal{O}_v$ is $\exists$-definable in $F$ for each $v \in S$.
Then $\Phi_{(q_v)_{v \in S}}$ is $\exists$-definable in $F$.
\end{lem}
\begin{proof}
It suffices to show the statement when $S = \lbrace v \rbrace$, since the general case then follows by the fact that finite intersections of existentially definable sets are again existentially definable.
If $\mathcal{O}_v$ is $\exists$-definable in $F$, then so is $\mathfrak{m}_v$:
letting $\pi \in F$ be a uniformiser for $v$, we have for an element $x \in F$ that $x \in \mathfrak{m}_v$ if and only if $\frac{x}{\pi} \in \mathcal{O}_v$.
Also $\mathcal{O}_v^\times$ is then existentially definable.

Furthermore, when $a_1, \ldots, a_{d-1}, b, 1+4b \in \mathcal{O}_v^\times$, then, in view of \Cref{L:PfisterNontrivialSlot}, $\partial_v \llangle a_1, \ldots, a_d, b ]]_F = [q_v] + I_q^{d+1} (Fv)$ is possible only when $v(a_d)$ is odd.
The property that $v(a_d)$ is odd is equivalent to $a_d \in \pi F\pow{2} \cdot \mathcal{O}_v^\times$ and thus existentially definable.
Furthermore, in this case we have $\partial_v \llangle a_1, \ldots, a_d, b ]]_F = [\llangle \overline{a_1}, \ldots, \overline{a_{d-1}}, \overline{b}]]_{Fv}] + I_q^{d+1} (Fv)$ by \Cref{P:Pfisterresidue}.
We thus wish to express by an existential formula that the classes of
$\llangle \ovl{a_1}, \dotsc, \ovl{a_{d-1}}, \ovl{b}]]_{Fv}$ and $q_v$ in $I_q^d(Fv)/I_q^{d+1}(Fv)$ are equal.
This equality is equivalent to the condition that $\llangle \ovl{a_1}, \ldots, \ovl{a_{d-1}}, \ovl{b}]]_{Fv} \cong q_v$
by \cite[Corollary 23.9]{ElmanKarpenkoMerkurjev} (observing that both sides are Pfister forms and therefore represent the value $1$).
Isometry of the forms $\llangle \ovl{a_1}, \ldots, \ovl{a_{d-1}}, \ovl{b}]]_{Fv}$ and $q_v$
can clearly be expressed over $Fv$ by an existential $\Lar(Fv)$-formula $\rho$,
and this can be transformed to an existential formula over $F$ as follows:
each constant from $Fv$ appearing in $\rho$ is replaced by an arbitrary lift to $\mc{O}_v$,
for each quantifier $\exists y$ appearing in $\rho$ one adds the condition $y \in \mc{O}_v$,
every atomic formula $\ovl{f} = \ovl{g}$ appearing in $\rho$ is replaced by an equivalence modulo $\mf{m}_v$ (i.e.~$f-g \in \mf{m}_v$),
and every atomic formula $\ovl{f} \neq \ovl{g}$ is replaced by $f-g \in \mc{O}_v^\times$.
Since the side conditions $y \in \mc{O}_v$, $f-g \in \mf{m}_v$ and $f-g \in \mc{O}_v^\times$
are all expressible by existential formulas,
this transformation yields an existential formula over $F$.
\end{proof}

We will now establish some cases where the quadratic forms $(q_v)_{v \in S}$ can be chosen such that $\Phi_{(q_v)_{v \in S}}$ is a Pfister stratum.
These cases rely on properties of the map $$\oplus \partial_v : I^{d+1}_q F / I^{d+2}_q F \to \bigoplus_{v \in \mathcal{V}} I^{d}_q (Fv) / I^{d+1}_q (Fv)$$ obtained by taking the sum of the residue morphisms;
this map is well-defined when $\mc{V}$ has the finite support property by virtue of \Cref{R:finiteSupportQF}.
On the one hand, we require the property that all elements of $I^{d+1}_q F / I^{d+2}_q F$ are classes of $(d+1)$-fold Pfister forms.
\begin{defi}\label{D:Linkage}
For $d \in \nat^+$, we say that \emph{$d$-fold Pfister forms over $F$ are linked}\index{linked} if every element of $I^d_qF/I^{d+1}_qF$ is the class of a $d$-fold Pfister form.
\end{defi}
On the other hand, we will study the cokernel of the map $\oplus \partial_v$.
We will show the following abstract result, which we will then combine with the results of \Cref{sec:reciprocity}.
\begin{lem}\label{L:PivotPropertyCondition}
Assume that $(d+1)$-fold Pfister forms over $F$ are linked.
Consider the exact sequence
$$ I^{d+1}_q F/I^{d+2}_q F \xrightarrow{\oplus \partial_v} \bigoplus_{v \in \mathcal{V}} I^{d}_q (Fv) / I^{d+1}_q (Fv) \rightarrow \Coker(\oplus \partial_v) \rightarrow 0. $$
If $\Coker(\oplus \partial_v) \neq 0$ and every component map $I^{d}_q (Fv) / I^{d+1}_q (Fv) \to \Coker(\oplus \partial_v)$ is surjective, then for every $S \subseteq \mathcal{V}$ with $\lvert S \rvert \in 2\nat + 1$ and for each $v \in S$ there exists a $d$-fold Pfister form $q_v$ over $Fv$ such that $\Phi_{(q_v)_{v \in S}}$ is a Pfister stratum for $S$.
\end{lem}
\begin{proof}
Denote by $N$ the natural map $\bigoplus_{v \in \mathcal{V}} I^{d}_q (Fv) / I^{d+1}_q (Fv) \rightarrow \Coker(\oplus \partial_v)$.
Observe that $\Coker(\oplus \partial_v)$ is $2$-torsion, since $\bigoplus_{v \in \mathcal{V}} I^{d}_q (Fv) / I^{d+1}_q (Fv)$ is.
Let $S \subseteq \mc{V}$ with $\lvert S \rvert \in 2\nat + 1$.
Note that, since $(d+1)$-fold Pfister forms over $F$ are linked, and $\partial_v$ is surjective and maps the class of a $(d+1)$-fold Pfister form to the class of a $d$-fold Pfister form for any $v \in \mc{V}$ (see \Cref{C:residueHomomorphism}),
also $d$-fold Pfister forms over $Fv$ are linked for all $v \in \mc{V}$.
Fix $\alpha \in \Coker(\oplus \partial_v) \setminus \lbrace 0 \rbrace$, and for each $v \in S$, fix a $d$-fold Pfister form over $Fv$ such that $N([q_v]) = \alpha$.
We claim that $(q_v)_{v \in S}$ is the desired tuple of forms making $\Phi_{(q_v)_{v \in S}}$ into a Pfister stratum for $S$.

Consider a tuple $(a_1, \ldots, a_d, b) \in \Phi_{(q_v)_{v \in S}}$ and let $q = \llangle a_1, \ldots, a_d, b ]]_F$.
Since $\partial_v q = [q_v] + I_q^{d+1} (Fv) \neq 0$ in $I_q^d (Fv) / I_q^{d+1} (Fv)$ for $v \in S$, we have $S \subseteq \delta q$.
We also have $N(\partial_v q) = \alpha$ for $v \in S$.
As $$N\Big(\sum_{v \in \mc{V}} \partial_v q\Big) = (N \circ \oplus \partial_v)([q]) = 0 \neq \alpha = \lvert S \rvert \alpha = N\Big(\sum_{v \in S} \partial_v q\Big),$$
we must have  $\partial_w q \neq 0$ in $I_q^d (Fw) / I_q^{d+1} (Fw)$ for some $w \in \mc{V} \setminus S$, i.e.~$w \in \delta q$.
This concludes the proof that $\Phi_{(q_v)_{v \in S}}$ satisfies \eqref{it:propA} in \Cref{D:PfisterStratum}.

For \eqref{it:propB} in that definition, consider some $w \in \mc{V} \setminus S$.
Fix a $d$-fold Pfister form $q_w$ over $Fw$ such that $N([q_w]) = \alpha$.
Since $N(\sum_{v \in S \cup \lbrace w \rbrace} [q_v]) = (\lvert S \rvert + 1)\alpha = 0$,
$\sum_{v \in S \cup \{ w \}} [q_v]$ must be the image under $\oplus \partial_v$ of the class of some $(d+1)$-fold Pfister form $q$ over $F$.
By \Cref{P:Pfister-semilocal-presentation} we can find $a_1, \ldots, a_{d-1}, b \in \bigcap_{v \in S} \mathcal{O}_v^\times$
and $a_d \in F^\times$ such that $1+4b \in \bigcap_{v \in S}\mc{O}_v^\times$ and $q \cong \llangle a_1, \ldots, a_d, b]]_F$.
Then $(a_1, \ldots, a_d, b) \in \Phi_{(q_v)_{v \in S}}$ satisfies $\delta\llangle a_1, \dotsc, a_d, b]] = \delta q = S \cup \{ w \}$.
\end{proof}

We will now apply this lemma to the situation of a regular function field $F/K$ with $\mathcal{V} = \Val(F/K)$ and $d = \dim_2 K \in \nat^+$.
\begin{opm}\label{R:standingHypPfisterStrata-ff}
In this situation the standing hypotheses of this section are satisfied,
in that $\mathcal{V}$ has the finite support property by \Cref{E:finiteSupport},
and for all dyadic $v \in \mathcal{V}$ (which only occur in characteristic $2$)
we have $\dim_2(Fv) = \dim_2 K = d$ since $Fv/K$ is finite,
so $[Fv : (Fv)^{(2)}] = 2^{d-1} < 2^d$.
\end{opm}

\begin{prop}\label{P:PivotPropertyCondition-ff}
Let $K$ be a field, and $d = \dim_2 K$.
Suppose that $1 \leq d < \infty$ and $I^d_qK \neq 0$.
Let $F/K$ be a regular function field in one variable,
assume that $(d+1)$-fold Pfister forms over $F$ are linked,
and write $\mathcal{V} = \Val(F/K)$.
Then for every $S \subseteq \mathcal{V}$ with $\lvert S \rvert \in 2\nat + 1$ and for each $v \in S$ there exists an anisotropic $d$-fold Pfister form $q_v$ over $Fv$ such that $\Phi_{(q_v)_{v \in S}}$ is a Pfister stratum for $S$ (with respect to $\mathcal{V}$).
\end{prop}
\begin{proof}
By \Cref{T:Reciprocity-Main}, we have an exact sequence
\[ I_q^{d+1}F \xrightarrow{\oplus \partial_v} \bigoplus_{v \in \Val} I_q^d(Fv) \to I_q^d K .\]
As noted after the statement of the theorem,
by the dimension condition $d = \dim_2 K$,
we may identify $I_q^{d+1}F$ with $I_q^{d+1}F/I_q^{d+2}F$,
and similarly for the other terms.
By \Cref{lem:transfer-surj-kernel}, each component map $I_q^d(Fv) \to I_q^d K$
is surjective, so in particular we may identify $I_q^d K$
with the cokernel of
\[ I^{d+1}_q F/I^{d+2}_q F \xrightarrow{\oplus \partial_v} \bigoplus_{v \in \mathcal{V}} I^{d}_q (Fv) / I^{d+1}_q (Fv) .\]
Now \Cref{L:PivotPropertyCondition} applies.
\end{proof}

\begin{opm}\label{R:global-fields-Pfister-stratum}
Let $F$ be a non-real global field.
Let $\Val$ be the set of all $\zz$-valuations on $F$ and $S \subseteq \Val$ such that $\lvert S \rvert \in 2\nat + 1$.
Note that for every dyadic $v \in \Val$ we have $[Fv : (Fv)\pow{2}] = 1$.
Furthermore, $2$-fold Pfister forms are linked, $I_q^2 F \neq 0$, and $I_q^3F = 0$.
The Albert-Brauer-Hasse-Noether Theorem yields an exact sequence for $I_q^2 F$.
Similar to the proof of \Cref{P:PivotPropertyCondition-ff},
one can use this exact sequence together with \Cref{L:PivotPropertyCondition}
to obtain a Pfister stratum for $S$.
An analogous construction was used in \cite{DaansGlobal}, see e.g.~the definition of the set $\Phi_u^S$ on page $1864$.
\end{opm}

\section{Application to function fields in one variable}\label{sect:universalFunctionfield}
We will in this section consider function fields in one variable $F/K$.
We work with the set $\mc{V}(F/K)$ of all $\zz$-valuations on $F$ which are trivial on $K$.
For $d = \dim_2 K$ (in the sense of \Cref{sec:reciprocity}) and assuming $1 \leq d < \infty$,
we will apply the results of the previous section with $\mathcal{V} = \Val(F/K)$.
For a $(d+1)$-fold Pfister form $q$ over $F$, we will -- as we did there -- denote by $\delta q$ the set of valuations in $\mc{V}(F/K)$ for which $\partial_v q \neq 0$.
By \Cref{P:deltavsDelta}, this is equivalently the set of $v \in \mc{V}(F/K)$ such that $q$ is anisotropic over the henselisation $F_v$.
Note that with $\mathcal{V} = \Val(F/K)$ we are in the situation of the previous section, see \Cref{R:standingHypPfisterStrata-ff}.
In particular, the set $\delta q$ is always finite by \Cref{R:finiteSupportQF}.

For a function field in one variable $F/K$ with $1 \leq d = \dim_2 K < \infty$, consider now the following conditions:
\begin{enumerate}[$(A)$]
\item\label{hyp1} $I^d_q K \neq 0$,
\item\label{hyp2} $(d+1)$-fold Pfister forms over $F$ are linked,
\item\label{hyp3} For a $(d+1)$-fold Pfister form $q = \llangle a_1, \ldots, a_d, b]]_F$,
  the ring $\bigcap_{v \in \delta \llangle a_1, \ldots, a_{d}, b]]_F} \mathcal{O}_v$ is $\exists$-definable in $F$,
  uniformly in the parameters $a_1, \ldots, a_d, b$.
\end{enumerate}
These three conditions will be satisfied if $K$ is a non-real global field and for many naturally occurring large fields, see \Cref{EtoAglobalex} below.

\begin{lem}\label{L:O_vdefinablefunctionfields}
Let $F/K$ be a regular function field in one variable with $1 \leq d = \dim_2 K < \infty$.
Assume that \eqref{hyp1} and \eqref{hyp3} hold.
For every $v \in \mc{V}(F/K)$, $\mc{O}_v$ is $\exists$-definable and $\forall$-definable in $F$.
\end{lem}
\begin{proof}
By \Cref{lem:transfer-surj-kernel} there is a surjective homomorphism $I^d_q Fv \to I^d_q K \neq 0$, whereby in particular $I^d_q Fv \neq 0$, i.e.~there exists an anisotropic $d$-fold Pfister form over $Fv$.
By \Cref{P:Pfisterresidue} there exists a $(d+1)$-fold Pfister form $q = \llangle a_1, \dotsc, a_d, b]]_F$ such that $v \in \delta q$.
The ring $R = \bigcap_{v \in \delta q} \mathcal{O}_v$ is $\exists$-definable by assumption.
Now the $\exists$-definability of $\mc{O}_v$ and $\mf{m}_v$ follows from the $\exists$-definability of $R$ via \cite[Corollary~6.13]{Andromeda-1}
(where $\Delta_K q$ in the notation there is our $\delta q$ by \Cref{P:deltavsDelta}).
Finally, if $\mf{m}_v$ is $\exists$-definable, then $\mc{O}_v$ is $\forall$-definable by \Cref{P:EtoAgeneral}.
\end{proof}
\begin{stel}\label{EtoAglobal}
Let $F/K$ be a regular function field in one variable with $1 \leq d = \dim_2 K < \infty$.
Assume that (\ref{hyp1})-(\ref{hyp3}) hold.
Let $S \subseteq \mc{V}(F/K)$ be finite.
Then the ring of $S$-integers
\begin{displaymath}
\mathcal{O}_S = \bigcap_{v \in \mc{V}(F/K) \setminus S} \mc{O}_v = \lbrace x \in F \mid \forall v \in \mc{V}(F/K) \setminus S : v(x) \geq 0 \rbrace
\end{displaymath}
is $\forall$-definable in $F$.
\end{stel}
\begin{proof}
Let us first consider the case where $\lvert S \rvert$ is odd.
We want to invoke \Cref{T:UniversalDefinitionsGeneralTechnique} and need to verify that its hypotheses are satisfied.

By hypothesis \eqref{hyp2} $(d+1)$-fold Pfister forms over $F$ are linked, and by \eqref{hyp1} we have $I^d_q K/ I^{d+1}_q K \cong I^d_q K \neq 0$.
We conclude by \Cref{P:PivotPropertyCondition-ff} that we can find a set $\Phi \subseteq F^{d+1}$ which is a Pfister stratum for $S$ and which furthermore is existentially definable in view of \Cref{Phidefinable} and \Cref{L:O_vdefinablefunctionfields}.
We have verified the hypotheses of \Cref{T:UniversalDefinitionsGeneralTechnique} and conclude that $\mc{O}_S$ is universally definable in $F$.

If $\lvert S \rvert$ is even, then pick $v \in \mc{V}(F/K) \setminus S$ arbitrary and set $S' = S \cup \lbrace v \rbrace$.
By the case covered in the above paragraph, $\mc{O}_{S'}$ is $\forall$-definable in $F$. But then also $\mc{O}_S = \mc{O}_{S'} \cap \mc{O}_v$ is $\forall$-definable by virtue of \Cref{L:O_vdefinablefunctionfields}.
\end{proof}
\begin{opm}\label{R:global-fields-universal}
When $F$ is a non-real global field, then it satisfies \eqref{hyp3} for $d=1$ by the work of Poonen and Koenigsmann and subsequent generalisations, see e.g.~\cite[Lemma 5.2]{DaansGlobal} where this is explained in the language of quaternion algebras based on the work of \cite{Dit17}; the translation between quaternion algebras and $2$-fold Pfister forms is standard (see e.g.~\cite[Section 12]{ElmanKarpenkoMerkurjev}).

One can thus use the existence of a Pfister stratum as discussed in \Cref{R:global-fields-Pfister-stratum} together with the above fact to derive from \Cref{T:UniversalDefinitionsGeneralTechnique} the main result of the sequence of papers \cite{Koe16,Par13,Eis18,DaansGlobal},
namely, that rings of $S$-integers in $F$ are $\forall$-definable in $F$.
In this way, the techniques developed abstractly in \Cref{sect:UniversalDefinitionGeneralTechnique} can be seen as capturing the essence of this construction for global fields.
\end{opm}

To give examples where hypotheses \eqref{hyp1}-\eqref{hyp3} are satisfied, we introduce a class of fields.
For a natural number $d$, we denote by $\mc{K}_d$ the class of all fields $K$ with the following properties:
\begin{enumerate}[(i)]
\item\label{it:K_d-1} $\dim_2(K) = d$,
\item\label{it:K_d-2} there exists some finite extension $L/K$ such that $I_q^d L \neq 0$,
\item\label{it:K_d-3} for every function field in one variable $F/K$, $(d+1)$-fold Pfister forms over $F$ are linked.
\end{enumerate}
The property for a field $K$ with $\dim_2(K) = d$ that, for every function field in one variable $F/K$, $(d+1)$-fold Pfister forms over $F$ are linked, is studied in \cite{DaansLinkage}.
When $\charac(K) = 2$, then \eqref{it:K_d-3} follows from \eqref{it:K_d-1}, see \cite[Proposition 3.5]{DaansLinkage} (essentially by \cite{ChapmanMcKinnieSymbolLength}).
On the other hand, if $\charac(K) \neq 2$, then \eqref{it:K_d-2} follows from \eqref{it:K_d-1} by definition of $\dim_2(K) = \cd_2(K)$ and \Cref{P:cdCharacterisation}.

\begin{vbn}\label{E:linkage}
Let $K$ be a field, $d \in \nat$.
\begin{enumerate}
\item\label{it:algclosed} If $K$ is algebraically closed and $\charac(K) \neq 2$, then $K \in \mc{K}_0$.
\item\label{it:finite} If $K$ is finite, then $K \in \mc{K}_1$.
\item\label{it:global} If $K$ is a non-real global or non-archimedean local field, then $K \in \mc{K}_2$.
\item\label{it:PAC} If $K$ is pseudo-algebraically closed, has no inseparable quadratic extension, but has a separable extension of even degree, then $K \in \mc{K}_1$.
\item\label{it:Cd} If $\charac(K) \neq 2$, $K$ is a $C_d$-field (see e.g.~\cite[Section 21.2]{Fri08}) and $\cd_2(K) \geq d$, then $K \in \mc{K}_d$.
\end{enumerate}
\end{vbn}
\begin{proof}
\eqref{it:Cd} When $K$ is a $C_d$-field, then also every finite field extension $L/K$ is a $C_d$-field.
Furthermore, since a $(d+1)$-fold Pfister form over $L$ has $2^{d+1} > 2^d$ variables, it must always be isotropic. It follows that $I_q^{d+1} L = 0$.
In view of \Cref{P:cdCharacterisation}, $\cd_2(K) \leq d$; together with the hypothesis we obtain $\cd_2(K) = d$.
 It now suffices to show the linkage property. For this, see \cite[Examples 3.7(4)]{DaansLinkage}.
 
\eqref{it:algclosed} This follows from \eqref{it:Cd}, since algebraically closed fields are $C_0$-fields.

\eqref{it:finite} Since $K$ has a separable quadratic extension, we have $I_q^1 K \neq 0$.
The result now follows from \eqref{it:Cd} if $\charac(K) \neq 2$, since finite fields are $C_1$-fields.
If $\charac(K) = 2$, then $[K : K\pow{2}] = 1$ since finite fields are perfect, hence the result also follows.

\eqref{it:global} We have $I_q^2 K \neq 0$, as well as $I_q^3 L = 0$ for every finite extension $L/K$, by the classical theory of quadratic forms over global and local fields, see e.g.~\cite[Chapter VI]{OMe00}.
Furthermore, if $\charac(K) = 2$, then $K$ is a function field in one variable over a finite field,
whence $[K : K\pow{2}] = 2$, and $K \in \mc{K}_2$ follows.
If $\charac(K) > 2$, $K$ is a $C_2$-field, hence the fact that $K \in \mc{K}_2$ now follows from \eqref{it:Cd}.
If $\charac(K) = 0$, then $\cd_2(K) = 2$ by \Cref{P:cdCharacterisation},
and the desired linkage statement can be found in \cite[Examples 3.7(1),(2)]{DaansLinkage} (essentially proven in \cite{Suresh_ThirdGalCohom}).

\eqref{it:PAC} If $K$ is pseudo-algebraically closed, then $\cd_2(K) \leq 1$ \cite[Corollary 11.6.8]{Fri08}.
Since $K$ has a separable extension of even degree,
some finite extension $L/K$ has a separable quadratic extension, and so $I_q^1 L \neq 0$.
If $\charac(K) = 2$, then by the assumption that $K$ has no inseparable quadratic extension, $[K : K\pow{2}] = 1$.
If $\charac(K) \neq 2$, then the linkage property follows from \cite[Examples 3.7(3)]{DaansLinkage} (essentially proven in \cite{BGStrongLinkage}).
\end{proof}
\begin{prop}\label{P:linkage-stability}
Let $d,e \in \nat$, $K \in \mc{K}_d$.
\begin{enumerate}
\item\label{it:K_d-fin} If $L/K$ is finite, then $L \in \mc{K}_d$.
\item\label{it:K_d-fg} If $\charac(K) = 2$ and $L/K$ is a finitely generated extension of transcendence degree $e$, then $L \in \mc{K}_{d+e}$.
\item\label{it:K_d-CDV} If $(L, v)$ is a complete $\zz$-valued field with $Lv = K$, then $L \in \mc{K}_{d+1}$.
\item\label{it:K_d-HV} If $(L, v)$ is a henselian valued field, $v$ is non-dyadic, $Lv = K$, and $(vL : 2vL) = 2^e$, then $L \in \mc{K}_{d+e}$.
\end{enumerate}
\end{prop}
\begin{proof}
\eqref{it:K_d-fin}:
We previously observed that $\dim_2 L = \dim_2 K = d$.
If $K'/K$ is a finite extension with $I_q^d K' \neq 0$,
any finite extension $L'/L$ containing $K'$ will satisfy $I_q^d L' \neq 0$ by surjectivity of the transfer map $I_q^d L' \to I_q^d K'$ (\Cref{lem:transfer-surj-kernel}).
Lastly, every function field in one variable $F/L$ is in particular a function field in one variable over $K$,
and so $(d+1)$-fold Pfister forms over $F$ are linked.

\eqref{it:K_d-fg}:
Because of \eqref{it:K_d-fin}, it suffices to consider the case $L = K(T)$.
We previously observed that $\dim_2 L = \dim_2 K + 1$.
Let $K'/K$ be a finite extension with $I_q^d K' \neq 0$.
The finite extension $L' = K'(T)$ of $L$ satisfies $I_q^{d+1} L' \neq 0$:
This follows for example from the surjectivity of the residue map $\partial_v\colon I_q^{d+1} L'/I_q^{d+2} L' \to I_q^d K'/I_q^{d+1} K'$
(\Cref{C:residueHomomorphism}) for the $T$-adic valuation $v$ on $L'$ with residue field $K'$,
or from \Cref{T:Reciprocity-Main}.
The linkage statement for function fields over $L$ is automatic in characteristic $2$.

\eqref{it:K_d-CDV} and \eqref{it:K_d-HV}:
Let $(L, v)$ be as in either \eqref{it:K_d-CDV} or \eqref{it:K_d-HV}; in the first case, set $e = 1$.
Let $F/L$ be a function field in one variable.
If $\charac(L) \neq 2$, it follows from \cite[Theorem 6.3 and Theorem 6.6]{DaansLinkage} that $(d+e+1)$-fold Pfister forms over $F$ are linked and $(d+e+2)$-fold Pfister forms over $F$ are isotropic;
since the same holds for finite extensions of $F$,
we in particular have $\cd_2(F) \leq d+e+1$ by \Cref{P:cdCharacterisation},
so $\cd_2(L) = \cd_2(F) - 1 \leq d+e$.
If $\charac(L) = 2$, then we are in case \eqref{it:K_d-CDV}, and it follows by \cite[Proposition 6.8]{DaansLinkage} that $\dim_2(L) = d+1$.

It remains to show that $I_q^{d+e} L' \neq 0$ for some finite field extension $L'/L$.
After replacing $(L, v)$ by a finite extension if necessary, we may assume that $I_q^d(Lv) \neq 0$.
Now use that $I_q^{d+e} L$ is isomorphic to $I_q^d(Lv)$ via the residue homomorphism to obtain that $I_q^{d+e} L \neq 0$,
see \cite[Theorem 4.9]{DaansLinkage} (a generalisation of \Cref{C:residueHomomorphism}).
\end{proof}

\begin{lem}\label{L:hyp1andhyp2}
Let $d \in \nat^+$, $K \in \mc{K}_d$, and $F/K$ a function field in one variable, not necessarily regular.
Then there exists some finite field extension $L/K$ contained in the algebraic closure of $F$ such that $FL/L$ is a regular function field satisfying \eqref{hyp1} and \eqref{hyp2}.
\end{lem}
Here, the assumption that $L$ is contained in the algebraic closure of $F$ is there to have an unambiguously defined compositum $FL$.
\begin{proof}
By the assumption that $K \in \mc{K}_d$, there exists a finite field extension $L_0/K$ with $I^d_q L_0 \neq 0$.
We may take a further finite extension $L/L_0$ to obtain that the function field $FL/L$ is regular, by first taking the relative algebraic closure of $L_0$ in $FL_0$, followed by a finite purely inseparable extension, see e.g. \cite[Theorem 8.6.10]{VillaSalvador}.
We still have $I^d_q L \neq 0$ by the surjectivity of the transfer map $I^d_q L \to I^d_q L_0$ (see \Cref{lem:transfer-surj-kernel}), and $L \in \mc{K}_d$.
We conclude that $FL/L$ is a regular function field in one variable satisfying \eqref{hyp1}.
Furthermore, $FL$ satisfies \eqref{hyp2} by the definition of $L \in \mc{K}_d$.
\end{proof}

We can now state our main theorem on universal definability of rings of $S$-integers.
To state our results in the highest generality first, we will, as in \cite{Andromeda-1}, rely on the notion of large fields, as introduced by Pop.
A field $K$ is called \emph{large} (some authors prefer the term \emph{ample}) if every smooth curve over $K$ which has a $K$-rational point has infinitely many $K$-rational points.
We refer to \cite{Pop_LittleSurvey} for an overview of examples and basic properties.
We mention in particular that all complete and all non-trivially henselian valued fields are large,
as well as all pseudo-algebraically closed fields.
\begin{stel}\label{EtoAglobalex}
Assume that one of the following holds:
  \begin{enumerate}
  \item $K$ is a global field;
  \item some finite extension $K'/K$ is large and satisfies $K' \in \mc{K}_d$ for some $d \in \nat^+$.
  \end{enumerate}
  Let $F/K$ be a function field in one variable, $S \subseteq \mc{V}(F/K)$ be finite. The subring $\mathcal{O}_S$
is $\forall$-definable in $F$.
\end{stel}
\begin{proof}
If $K$ is a global field, let $K' = K(\sqrt{-1})$ and $d = 2$,
so $K' \in \mc{K}_d$ in both cases by \Cref{E:linkage}\eqref{it:global}.
By \Cref{L:hyp1andhyp2} we can find a further finite extension $L/K'$ such that
$FL/L$ is regular and satisfies \eqref{hyp1} and \eqref{hyp2}.
Furthermore, $L$ is either a non-real global field or a large field,
so \eqref{hyp3} is satisfied for $FL/L$ by \cite[Theorem 10.13]{Andromeda-1} in the global case
and \cite[Theorem 9.17]{Andromeda-1} in the large case
(observing as in the proof of \Cref{L:O_vdefinablefunctionfields} that for a Pfister form $q = \llangle a_1, \dotsc, a_d, b]]_F$,
the ring $\bigcap_{v \in \delta q} \mathcal{O}_v$ is notated $\mathcal{O}(\Delta_K q)$ in \cite{Andromeda-1},
using that $\Delta_K$ there agrees with our $\delta q$ by \Cref{P:deltavsDelta}).

By applying \Cref{EtoAglobal} to the function field $FL/L$ and to the finite set 
$$S' = \lbrace w \in \mc{V}(FL/L) \mid \exists v \in S : \mc{O}_v = \mc{O}_w \cap F \rbrace,$$
we find that $\mc{O}_{S'}$ is universally definable in $FL$.
Furthermore, we have $\mc{O}_{S'} \cap F = \mathcal{O}_S$.
We thus obtain the required universal definability of $\mathcal{O}_S$ in $F$ by a standard interpretation argument,
see e.g.~\cite[Lemma~6.3]{MillerShlapentokh_v2}.
\end{proof}

\begin{gev}\label{C:EtoAglobalex}
Assume that one of the following holds:
\begin{itemize}
\item $K$ is pseudo-algebraically closed, has no inseparable quadratic extension, but has a separable extension of even degree,
\item $K$ carries a complete $\zz$-valuation with residue field in $\mc{K}_d$ for some $d \in \nat$.
\item $K$ carries a henselian non-dyadic valuation $v$ with residue field in $\mc{K}_d$ for some $d \in \nat$ and such that $1 < (vK : 2vK) < \infty$.
\end{itemize}
Let $F/K$ be a function field in one variable, $S \subseteq \mc{V}(F/K)$ be finite. The subring
$\mathcal{O}_S$
is $\forall$-definable in $F$.
\end{gev}
\begin{proof}
It follows from \Cref{E:linkage} and \Cref{P:linkage-stability} that $K \in \mc{K}_1$ in the first case, that $K \in \mc{K}_{d+1}$ in the second case, and that $K \in \mc{K}_{d+e}$ for $e = \log_2((vK : 2vK))$ in the third case.
Furthermore, in each case, $K$ is large \cite[2]{Pop_LittleSurvey}.
The rest of the statement thus falls under the second case of \Cref{EtoAglobalex}.
\end{proof}

Let us now discuss a consequence of \Cref{EtoAglobalex} for definability of very general relations in function fields over number fields.
\begin{gev}\label{ce-sets-ea-definable}
  Let $K$ be a number field and $F/K$ be a function field in one variable.
  Let $k \geq 1$.
  Every computably enumerable subset $A \subseteq F^k$ is $\exists\forall$-definable in $F$.
\end{gev}
Let us clarify terminology.
The field $F$ is finitely generated over $\mathbb{Q}$, and therefore admits a bijection $\mathbb{N} \to F$
such that the pullback of the graphs of addition and multiplication gives relations on $\mathbb{N}$
which are computable in the usual sense;
in other words, $F$ can be seen as a computable field.
A subset $A \subseteq F^k$ is \emph{computably enumerable} if its pullback with respect to the chosen bijection $\mathbb{N} \to F$
is a computably enumerable subset of $\mathbb{N}^k$.
This notion does not in fact depend on the choice of bijection $\mathbb{N} \to F$;
this is a consequence of the fact that finitely generated fields
are \emph{computably stable}.
Similarly, we have a well-behaved notion of computably enumerable subset of $\mc{O}_S^k$
for every finite $S \subseteq \Val(F/K)$,
since the finitely generated integral $\qq$-algebra $\mc{O}_S$ is likewise computably stable.
See the brief \cite[Subsection 1.1]{Demeyer_DiophFuncFldsGlobal} for all this,
using the adjective \emph{recursive} instead of \emph{computable};
see also the extensive survey \cite[Sections 2 \& 3]{Handbook-CompRingsFields}
for more on computable and computably stable rings.
The notion of listable sets of \cite{Pasten} is the same as what we call computably enumerable here,
and the relevant fact about $F$ and $\mc{O}_S$
is that they are uniquely listable in the sense introduced there,
which is related (but not identical) to computable stability.
\begin{proof}[Proof of \Cref{ce-sets-ea-definable}.]
  We follow the proof of \cite[Proposition 6.1]{Daans_Defining10}.
  Let $S \subseteq \Val(F/K)$ be an arbitrary non-empty finite set.
  Then the $K$-subalgebra $\mc{O}_S$ of $F$ has fraction field $F$.
  We consider the set
  \[ B = \{ (a_1, \dotsc, a_k, b) \in \mc{O}_S^{k+1} \mid b \neq 0, (\frac{a_1}{b}, \dotsc, \frac{a_k}{b}) \in A \}. \]
  Given the assumption that $A$ is computably enumerable, $B$ is also computably enumerable.

  Every computably enumerable subset of $\mc{O}_S^{k+1}$ is existentially definable in $\mc{O}_S$:
  in \cite[Theorem 7.1]{MillerShlapentokh_v2} this is stated for $k=0$, but the proof makes it clear that the stronger statement holds, as in the special case given in \cite{Demeyer_DiophFuncFldsGlobal} and serving as the motivation of \cite[Theorem 7.1]{MillerShlapentokh_v2}.
  In particular, $B$ is existentially definable in $\mc{O}_S$,
  so there exists $m \geq 0$ and a quantifier-free $\Lar(\mc{O}_S)$-formula $\varphi(X_1, \dotsc, X_{k+1}, Y_1, \dotsc, Y_m)$
  such that \[ B = \big\{ (a_1, \dotsc, a_k, b) \in \mc{O}_S^{k+1} \mid \exists c_1, \dotsc, c_m \in \mc{O}_S \big(\varphi(a_1, \dotsc, a_k, b, c_1, \dotsc, c_m)\big) \big\} .\]
  It follows that an arbitrary tuple $(d_1, \dotsc, d_k) \in F^k$ lies in $A$ if and only if 
  \begin{multline*}
    \exists b, c_1, \dotsc, c_m \in F \Big(b \neq 0 \wedge b \in \mc{O}_S \wedge \bigwedge_{i=1}^k d_i b \in \mc{O}_S \wedge \bigwedge_{i=1}^m c_i \in \mc{O}_S  \\
    \wedge \varphi(d_1 b, \dotsc, d_k b, b, c_1, \dotsc, c_m) \Big) .
  \end{multline*}
  Since $\mc{O}_S$ is $\forall$-definable in $F$ by \Cref{EtoAglobalex},
  we see that $A$ is $\exists\forall$-definable in $F$.
\end{proof}

Results analogous to \Cref{ce-sets-ea-definable} also hold in some other settings,
for instance for rational function fields $F = K(T)$ over some large infinite algebraic extensions $K/\qq$
like $\qq_p \cap \overline{\qq}$ for an odd prime $p$,
by applying the results of \cite{DegrooteDemeyer-14} to the ring $K[T]$ instead of \cite[Theorem 7.1]{MillerShlapentokh_v2}.
We do not pursue this further here.

When $F$ is a function field in one variable over a field $K$, one can verify that the rings of the form $\mc{O}_S = \bigcap_{v \in \mc{V}(F/K) \setminus S} \mathcal{O}_v$ for some finite set $S$ are precisely the integrally closed, finitely generated $K$-subalgebras of $F$.
We conclude this section with a short discussion on universal definability of finitely generated $K$-subalgebras of $F$ which are not necessarily integrally closed.

\begin{prop}\label{propfingen}
  Let $K$ be a field and let $F/K$ be a function field in one variable such that
\begin{itemize}
\item $K$ is $\exists$-definable in $F$,
\item for any $v \in \mc{V}(F/K)$, $\mathcal{O}_v$ is $\exists$-definable in $F$,
\item for any finite subset $S \subseteq \mc{V}(F/K)$, $\mathcal{O}_S$ is $\forall$-definable in $F$.
\end{itemize}  
Then any finitely generated $K$-subalgebra of $F$ having $F$ as its fraction field is $\forall$-definable in $F$.
\end{prop}
\begin{proof}
Let $R$ be a finitely generated $K$-subalgebra of $F$ with fraction field $F$, and let $R'$ be the integral closure of $R$ in $F$.
Then $R$ is noetherian and has Krull dimension $1$ \autocite[Theorem 8.A]{Eis04},
and $R'$ is finitely generated as an $R$-module \autocite[Corollary 13.13]{Eis04},
in particular finitely generated as a $K$-algebra.
Now \cite[Proposition 6.4]{Andromeda-1} implies that $R' = \mc{O}_S$ for some finite set $S \subseteq \mc{V}(F/K)$,
so $R'$ is universally definable in $F$ by assumption.

Letting $r \in R\setminus \lbrace 0 \rbrace$ be a common denominator of a set of generators of $R'$ as an $R$-module, we have that $rR' \subseteq R$ and $R'/rR'$ has Krull dimension $0$. Since $R'/rR'$ is finitely generated as a $K$-algebra, it is therefore a finite-dimensional $K$-vector space by Noether Normalisation \autocite[Theorem 8.A1]{Eis04}.

Pick elements $b_1, \ldots, b_m \in R'$ such that their residues form a $K$-basis of $R'/rR'$.
A given $x \in F$ lies in $R$ if and only if $x \in R'$ and the element $x + rR' \in R'/rR'$,
which can be uniquely expressed modulo $rR'$ as a $K$-linear combination of the $b_j$,
lies in $R/rR'$.
Therefore $x \in F$ lies in $R$ if and only if $x \in R'$ and for all $y_1, \ldots, y_m \in F$ one has
\begin{displaymath}
(y_1, \ldots, y_m \in K \text{ and } x - \sum_{j=1}^m y_jb_j \in rR') \rightarrow \sum_{j=1}^m y_jb_j +rR' \in R/rR'.
\end{displaymath}
To argue that this condition can be described with a universal formula,
it suffices to establish that the antecedent of the implication can be expressed by an existential formula,
and the consequent by a quantifier-free formula.
The set $R/rR'$ is a $K$-subspace of the finite-dimensional $K$-vector space $R'/rR'$, so the set of $(y_1, \ldots, y_m) \in K^m$ satisfying $\sum_{j=1}^m y_jb_j + rR' \in R/rR'$ is a subspace of $K^m$ and therefore quantifier-freely characterised by the vanishing of finitely many linear forms on $K^m$.

The field $K$ is existentially definable in $F$ by assumption.
To express that for $x \in R'$ and $y_1, \ldots, y_m \in K$ one has $x - \sum_{j=1}^m y_jb_j \in rR'$ , it suffices to state that for each of the finitely many maximal ideals $\mathfrak{p}$ of $R'$ containing $r$ one has $x - \sum_{j=1}^m y_jb_j \in rR'_{\mathfrak{p}}$, where $R'_{\mathfrak{p}}$ is the localisation of $R'$ at $\mathfrak{p}$.
But $R'_{\mathfrak{p}}$ is a valuation ring of $F$ containing $K$, hence existentially definable in $F$ by assumption.
\end{proof}
\begin{stel}\label{T:UniversalFunctionFieldMain}
Let $K$ be a field, and assume that some finite extension $K'/K$ is large and $K' \in \mc{K}_d$ for some $d \in \nat^+$.
For instance, it suffices that one of the following holds:
\begin{itemize}
\item $K'$ is a non-archimedean local field,
\item $K'$ is pseudo-algebraically closed, has no inseparable quadratic extension, but has a separable extension of even degree,
\item $K'$ carries a complete $\zz$-valuation with residue field in $\mc{K}_e$ for some $e \in \nat$.
\item $K'$ carries a henselian non-dyadic valuation $v$ with residue field in $\mc{K}_e$ for some $e \in \nat$ and such that $1 < (vK' : 2vK') < \infty$.
\end{itemize}
Let $F/K$ be a regular function field in one variable. Any finitely generated $K$-subalgebra of $F$ having $F$ as its fraction field is $\forall$-definable in $F$.
\end{stel}
\begin{proof}
This follows from \Cref{propfingen} once we verify that the assumptions are satisfied.
Consider first the regular function field in one variable $FK'/K'$.
Since $K'$ is large, it is existentially definable in $FK'$ \cite[Theorem 2]{Koe02}.
The universal definability of rings of $S$-integers in $FK'/K'$ follows from \Cref{EtoAglobalex}.
The existential definability of $\mathcal{O}_v \subseteq FK'$ for every $v \in \Val(FK'/K')$ follows from \cite[Theorem 9.12]{Andromeda-1} (with $K_0 = K'$, $F_0 = FK'$).
This shows that $FK'/K'$ satisfies the assumptions of \Cref{propfingen}.

To see that also $F/K$ then satisfies the hypotheses of \Cref{propfingen}, we argue as in the second paragraph of the proof of \Cref{EtoAglobalex}. 
Observe that $K = K' \cap F$, that for every $v \in \Val(F/K)$ there exists $v' \in \Val(FK'/K')$ such that $\mc{O}_v = \mc{O}_{v'} \cap F$, and that for every finite subset $S \subseteq \Val(F/K)$, there exists a finite subset $S' \subseteq \Val(FK'/K')$ such that $\mc{O}_S = \mc{O}_{S'} \cap F$.
The existential, respectively universal, definability of the sets $K$, $\mc{O}_v$ and $\mc{O}_S$ in $F$ now follows from the corresponding definability of $K'$, $\mc{O}_{v'}$ and $\mc{O}_{S'}$ in $FK'$ by a standard interpretation argument, see \cite[Lemma~6.3]{MillerShlapentokh_v2}.
\end{proof}

\begin{vbn}\label{E:large-basefields-definability}
For each of the following fields $K$, we have by \Cref{T:UniversalFunctionFieldMain} that, for every regular function field in one variable $F/K$, every finitely generated $K$-subalgebra of $F$ having $F$ as its fraction field is $\forall$-definable in $F$.
\begin{itemize}
\item $K = \rr(\!(T)\!)$ or $K = \cc(\!(T)\!)$, the field of formal Laurent series with real, respectively complex, coefficients.
  Here it suffices to observe that $\cc(\!(T)\!)$ is a complete $\zz$-valued field with residue field $\cc \in \mc{K}_0$;
  see \Cref{E:linkage}.
\item $K$ is a higher local field, i.e.~there exists a sequence $K_1, K_2, \ldots, K_d = K$ where $d > 1$, $K_1$ is a finite field, and $K_{i+1}$ is a complete $\zz$-valued field with residue field $K_i$.
By \Cref{E:linkage} we have $K_1 \in \mc{K}_1$ and then by \Cref{P:linkage-stability} we see that $K \in \mc{K}_d$.
\item $K$ is a pseudo-finite field, i.e.~it is elementarily equivalent to an ultraproduct of finite fields.
See \cite[Section 20.10]{Fri08} for a discussion of pseudo-finite fields.
A pseudo-finite field is in particular always pseudo-algebraically closed and perfect, and has a separable quadratic extension, so lies in $\mc{K}_1$ by \Cref{E:linkage}\eqref{it:PAC}.
\end{itemize}
\end{vbn}

\begin{opm}\label{R:elim-nonconst-consts}
  Applying \Cref{EtoAglobal} with $S = \emptyset$ yields that $K$ is defined in $F$ by a universal $\Lar(F)$-formula whenever the hypotheses are satisfied.
  On the other hand, note that if $C$ is a curve over a field $K$ for which the set of $K$-rational points $C(K)$ is infinite,
  then $K$ cannot be defined by a universal $\Lar(K)$-formula in the function field $K(C)$ since $K$ is existentially closed in $K(C)$ \autocite[Fact 2.3]{Pop_LittleSurvey}.

  If $C(K)$ is finite, different behaviour can be observed.
  For instance, if $K$ is of characteristic 0, $C/K$ is a curve of genus greater than $1$ with $C(K)$ finite and we set $F = K(C)$,
  then also $C(F) \setminus C(K)$ is finite (cf.\ the proof of \autocite[Theorem 1(3)]{Koe02}).
  By mapping $C(F) \setminus C(K)$ to $F$ with a suitable rational function,
  we obtain a non-empty existentially $\Lar(K)$-definable finite subset $A \subseteq F$ disjoint from $K$.
  This can be used to find an existentially $\Lar(K)$-definable element $T \in F \setminus K$, by applying a symmetric function to the elements of $A$.

  Let $s$ be a generator of $F$ over $K(T)$.
  Then an existential $\Lar(F)$-formula defining $F \setminus K$ in $F$ can be turned into
  an existential $\Lar(K)$-formula which also defines $F \setminus K$ in $F$,
  by writing all constants as rational functions of $s$ and $T$,
  and eliminating $s$ by existentially quantifying over roots of the minimal polynomial of $s$ over $K(T)$.
\end{opm}

\printbibliography
\end{document}